\documentclass[a4paper,11pt]{article}
\usepackage{graphicx}
\usepackage{amssymb}
\usepackage[french, english]{babel}
\usepackage[utf8]{inputenc}
\usepackage{amsmath,amsfonts,amssymb,amsthm}
\usepackage[T1]{fontenc}
\usepackage{fullpage}
\usepackage{hyperref}
\usepackage{mathrsfs}
\usepackage{relsize}
\usepackage{tikz}
\usepackage{bbm}
\usepackage{appendix}

\newcommand{\eps}{\varepsilon}

\newcommand{\N}{\mathbb N}

\newcommand{\R}{\mathbb R}

\newcommand{\E}{\mathbb E}

\renewcommand{\P}{\mathbb P}
\newcommand{\p}{\boldsymbol{\mathcal{P}}}
\newcommand{\hp}{\boldsymbol{\mathcal{P}}^h}

\newcommand{\argsh}{\mathrm{argsh}\,}

\theoremstyle{definition}

\newtheorem{thm}{Theorem}

\newtheorem{defn}{Definition}
\newtheorem{rem}[defn]{Remark}

\newtheorem{prop}[defn]{Proposition}
\newtheorem{corr}[defn]{Corollary}

\newtheorem{lem}[defn]{Lemma}

\tikzstyle{every node}=[circle, draw, fill=black!50, inner sep=0pt, minimum width=4pt]
\tikzstyle{rouge}=[circle, draw, fill=red, inner sep=0pt, minimum width=6pt]
\tikzstyle{bleu}=[circle, draw, fill=blue, inner sep=0pt, minimum width=6pt]
\tikzstyle{petitrouge}=[circle, draw, fill=red, inner sep=0pt, minimum width=4pt]
\tikzstyle{petitbleu}=[circle, draw, fill=blue, inner sep=0pt, minimum width=4pt]
\tikzstyle{texte}=[draw=none, fill=none]

\title{\bf{The Hyperbolic Brownian Plane}}
\author{Thomas \bsc{Budzinski} \footnote{ENS Paris and Université Paris-Saclay, \url{thomas.budzinski@ens.fr}}}
\date{}

\begin{document}

\maketitle

\begin{abstract} We introduce and study a new random surface, which we call the \emph{hyperbolic Brownian plane} and which is the near-critical scaling limit of the hyperbolic triangulations constructed in \cite{CurPSHIT}. The law of the hyperbolic Brownian plane is obtained after biasing the law of the Brownian plane \cite{CLGplane} by an explicit martingale depending on its perimeter and volume processes studied in \cite{CLGHull}. Although the hyperbolic Brownian plane has the same local properties as those of the Brownian plane, its large scale structure is much different since we prove e.g.~that is has exponential volume growth.
\end{abstract}

\section*{Introduction}

The construction and the study of random surfaces as scaling limits of random planar maps has been a very active field of research in the last years, see \cite{LeGallICM,MieStFlour} for survey. The first such random surface that was built is the Brownian map \cite{LG11,Mie11}, which is now known to be the scaling limit of a wide class of finite planar maps conditioned to be large \cite{Ab13,AA13,BLG13,BJM13,CLGmodif}. Curien \& Le Gall introduced the Brownian plane in \cite{CLGplane}, which can be seen as a non-compact version of the Brownian map. They showed that it is the scaling limit of the Uniform Infinite Planar Quadrangulation (UIPQ). They also conjectured it to be the scaling limit of several other random infinite lattices such as the Uniform Infinite Planar Triangulation (UIPT) of Angel \& Schramm \cite{AS03} (we verify this fact below for type-I triangulations). The goal of this paper is to introduce and to study a new random surface which we call \emph{the hyperbolic Brownian plane}. This surface is obtained as a near-critical scaling limit of the hyperbolic triangulations of \cite{CurPSHIT}.

\paragraph{The Brownian plane as the near-critical limit of the PSHT.} The spatial Markov property of random maps is a key feature of random lattices like the UIPT and UIPQ and has been used a lot in recent years to study their geometric structure, see e.g.~\cite{Ang03,BC16,CLGpeeling}. Recently, Angel \& Ray characterized all the triangulations of the half-plane enjoying a spatial Markov property and discovered a new family of triangulations of the half-plane having hyperbolic flavor \cite{AR13}. This has been extended to cover the case of the full-plane in \cite{CurPSHIT}. More precisely,  \cite{CurPSHIT} constructs a one-parameter family $( \mathbf{T}_{\kappa})_{0<\kappa \leq \kappa_c}$ of Markovian random triangulations of the plane, where the value $\kappa_{c}$ is equal to $\frac{2}{27}$. The triangulations $\mathbf{T}_{\kappa}$ for $\kappa<\kappa_c$ are called type-II Planar Stochastic Hyperbolic Triangulations (PSHT). At the critical value $\kappa = \kappa_{c}$, the random triangulation $ \mathbf{T}_{\kappa_{c}}$ is the UIPT of Angel \& Schramm, whereas $ \mathbf{T}_{\kappa}$ has hyperbolic features when $\kappa < \kappa_{c}$. Note that if $\kappa < \kappa_{c}$ is fixed, then it is impossible to rescale $ \mathbf{T}_{\kappa}$ to get a scaling limit in the Gromov--Hausdorff sense\footnote{we can find in the ball of radius $r$ of $ \mathbf{T}_{\kappa}$ a number of points at distance at least $\frac{r}{10}$ from each other that goes to $+\infty$ as $r \to +\infty$, so the sequence $\big( \frac{1}{r} B_r(\mathbf{T}_{\kappa}) \big)_{r \geq 1}$ is not tight for the Gromov--Hausdorff topology}. Hence, in order to get a proper scaling limit, it is necessary to let the parameter $\kappa \to  \kappa_{c}$ at the right speed as we renormalize the distances. If we let $\kappa \to \kappa_{c}$ too slow, then there is no scaling limit as above and if $\kappa \to \kappa_{c}$ too fast, then the scaling limit is just the Brownian plane: our approach is \textit{near-critical}.

Our main tool for proving such a convergence will be the absolute continuity relations between the hyperbolic triangulations and the UIPT. These relations allow us to deduce convergence results for hyperbolic maps from the analogous results for the UIPT. The above works \cite{AS03,CurPSHIT} deal with type-II triangulations where loops are forbidden. Unfortunately, as of today, no scaling limit result is available in the literature for type-II triangulations. This forces us to work with type-I triangulations (i.e. where loops are allowed), for which the convergence to the Brownian map has been established \cite{LG11}. Our first (easy) task is then to generalize the results of \cite{CurPSHIT} and to introduce the type-I PSHT, which we denote by $$\big( \mathbb{T}_\lambda \big)_{0 < \lambda \leq \lambda_{c}}, \quad  \mbox{where}\quad  \lambda_{c}= \frac{1}{12 \sqrt{3}}.$$ As above, in the critical case $\lambda= \lambda_{c}$, the random lattice $ \mathbb{T}_{\lambda_{c}}$ is just the type-I UIPT \cite{CLGmodif,St14}.

We denote by $\p$ the Brownian plane of \cite{CLGplane}. We recall that it is equipped with a volume measure $\mu_{\p}$. For $r \geq 0$, we write $B_r(\p)$ for the closed ball of radius $r$ centered at the origin point of $\p$. We also write $\overline{B_r}(\p)$ for the hull of radius $r$, that is, the union of $B_r(\p)$ together with all the bounded connected components of its complementary. We equip $\overline{B_r}(\p)$ with the induced metric (i.e. the restriction of the metric on $\p$), and the restriction of $\mu_{\p}$. We will therefore consider $\overline{B_r}(\p)$ as a measured metric space.
The last ingredients we need before stating our main theorem are the perimeter and volume processes of $\p$ introduced by Curien \& Le Gall \cite{CLGHull}. If $A \subset \p$, we will write $|A|=\mu_{\p}(A)$ for the measure of $A$. For $r>0$, the \textit{volume} of the hull of radius $r$ of $\p$ is $|\overline{B_r}(\p)|$. Moreover, Proposition 1.1 of \cite{CLGHull} states that the limit
\begin{equation}\label{perimeter_as_limit}
\lim_{\eps \to 0} \frac{1}{\eps^2} \left| B_{r+\eps} (\p) \backslash \overline{B_r}(\p) \right|
\end{equation}
is a.s well-defined and positive. We call it the \textit{perimeter} of $\overline{B_r}(\p)$ and denote it by $|\partial \overline{B_r}(\p)|$. Our main theorem is the following.

\begin{thm}[$\hp$ as a near-critical scaling limit of the PSHT]\label{mainthm} For $n \geq 0$, consider $ \mathbb{T}_{\lambda_{n}}$ the type-I planar stochastic hyperbolic triangulation of parameter $\lambda_{n} \to \lambda_{c}$ in such a way that  
\begin{equation} \label{lambdaspeed}
\lambda_n = \lambda_{c} \left( 1 - \frac{2}{3 n^4} \right)+ o \left( \frac{1}{n^4} \right).
\end{equation}
Then we have the following convergence for the local Gromov--Hausdorff--Prokhorov distance:
\[ \frac{1}{n}  \mathbb{T}_{\lambda_n} \xrightarrow[n\to\infty]{(d)} \hp,\]
where $\hp$ is a random locally compact metric space that we call the \textit{hyperbolic Brownian plane}. Its distribution is characterized by the fact that for every $r \geq 0$, the random measured metric space $\overline{B_r}(\hp)$ has density
\begin{equation}\label{densityph}
e^{-2|\overline{B_{2r}} (\p)|} e^{|\partial \overline{B_{2r}} (\p)|} \int_0^1 e^{-3 |\partial \overline{B_{2r}} (\p)| x^2}  \mathrm{d}x
\end{equation}
with respect to $\overline{B_r}(\p)$.
\end{thm}
The choice of the constant $\frac{2}{3}$ in \eqref{lambdaspeed} was made so that the expression \eqref{densityph} looks as simple as possible. Of course another choice woud have resulted in a scaling limit just obtained by dilating $\hp$. The fact that we need to bias $\overline{B_{r}}( \p)$ by a function of the perimeter and volume of the hull of radius $2r$ instead of $r$ may seem surprising. It is due to the fact that we equip $\overline{B_{r}}( \p)$ with the \emph{induced} distance. Hence, the distance between two points in the hull of radius $r$ of a map may depend on the part of the map that lies outside of this hull (but not outside of the hull of radius $2r$). In order to obtain a similar result with $| \partial \overline{B_r}(\p)|$ and $|\overline{B_r}(\p)|$, we would need to equip $\overline{B_{r}}( \p)$ with its \emph{intrinsic} distance instead of the \emph{induced} one (see Section 2.1 for more details about this distinction). We would also need to prove an analog of Proposition \ref{GHP} for the intrinsic distance, which we have not been able to do.

In the study of a one-parameter family of models exhibiting a critical behavior, it has become quite usual to study \emph{near-critical} (scaling) limits. By near-critical, we mean that the parameter converges to its critical value at the right speed as the distances in the graph or the mesh of the lattice are going to zero. Understanding the near-critical limit usually sheds some light on the critical model because of the existence of scaling relations between near-critical and critical exponents. See for example the works on near-critical percolation \cite{GPS13,Kes87,NW09}, on the Ising model \cite{DCGP14} or on the Erd\"os-R\'enyi random graph \cite{ABBG12,ABBGM13}. 

\paragraph{Techniques.} As said above, the idea of the proof of Theorem \ref{mainthm} is to use the absolute continuity relations between the hulls of the type-I UIPT and of the hyperbolic triangulations $ \mathbb{T}_{\lambda}$. Our main technical tool is a reinforcement of the convergence of the UIPT towards the Brownian plane.  In the result below, $| \overline{B}_{r}( \mathbb{T}_{\lambda_{c}})|$ and $| \partial \overline{B}_{r}( \mathbb{T}_{\lambda_{c}})|$ respectively stand for the volume (number of vertices) and perimeter of the hull of radius $r$ in the UIPT.

\begin{thm}[Extended convergence towards the Brownian plane] \label{joint}
We have the joint convergences
\[\Bigg( \frac{1}{n}  \mathbb{T}_{\lambda_{c}}, \Big(\frac{1}{n^4} |\overline{B_{rn}}( \mathbb{T}_{\lambda_{c}})| \Big)_{r \geq 0}, \Big( \frac{1}{n^2} |\partial \overline{B_{rn}}( \mathbb{T}_{\lambda_{c}})| \Big)_{r \geq 0} \Bigg) \xrightarrow[n\to\infty]{(d)} \Bigg( \p, \Big( 3 |\overline{B_r}(\p)| \Big)_{r \geq 0}, \Big( \frac{3}{2} |\partial \overline{B_r}(\p)| \Big)_{r \geq 0} \Bigg)\]
in the local Gromov--Hausdorff--Prokhorov sense for the first marginal and in the Skorokhod sense for the last two.
\end{thm}

The convergence of the first marginal has been established in the case of quadrangulations in \cite{CLGplane} and we extend the proof to cover our case. On the other hand, the joint convergence of the last two marginals follows from the work \cite{CLGpeeling} (both in the case of quadrangulations and type-I triangulations). But it is important in our Theorem \ref{joint} that those convergences hold jointly, which requires some additional work. The convergence of near-critical PSHT towards the hyperbolic Brownian plane then follows from Theorem \ref{joint} and a couple of asymptotic enumeration results gathered in Section 1.

\paragraph{Properties of $\hp$.}
We also establish some properties of the hyperbolic Brownian plane. Since the density \eqref{densityph} goes to $1$ as $r$ goes to $0$, the hyperbolic Brownian plane is "locally isometric" to the Brownian plane (and hence also to the Brownian map). More precisely, for all $\eps>0$, there is a $\delta>0$ and a coupling between $\p$ and $\hp$ such that, with probability at least $1-\eps$, they have the same ball of radius $\delta$ around the origin. We also prove that $\hp$ almost surely has Hausdorff dimension $4$ and is homeomorphic to the plane.

The Brownian map is known to be invariant under uniform re-rooting. This means that, if we resample its root uniformly according to its volume measure, the rooted metric space we obtain has the same distribution as the Brownian map. This property has played an important role in the proof of universality results in \cite{Ab13,AA13,BLG13,BJM13,LG11}, and in the axiomatic characterization of the Brownian map given by \cite{MS15}. Unfortunately, it makes no sense anymore when the volume measure is infinite. However, we prove the following property of $\hp$: for every measurable, nonnegative function $f$ we have
\[ \mathbb{E} \Big[ \int_{\hp} f(\hp, \rho, y) \mu_{\hp}(\mathrm{d}y) \Big] = \mathbb{E} \Big[ \int_{\hp} f(\hp, y, \rho) \mu_{\hp}(\mathrm{d} y) \Big], \]
where $\rho$ is the origin of $\hp$ and $\mu_{\hp}$ its volume measure. This property is a continuous analog of the discrete property of unimodularity, which is a natural substitute for infinite random graphs to invariance under uniform rerooting (see for example \cite{AL07} for the discrete case). More precisely, the two properties are equivalent for finite random graphs. Our result shows that the hyperbolic Brownian plane is a natural surface to look at even from the purely continuum point of view.

Also, since the hyperbolic Brownian plane is a biased version of the Brownian plane, it is possible to define the perimeter $P_r^h$ and the volume $V_r^h$ of its hull of radius $r$ as Curien and Le Gall did for the Brownian plane in \cite{CLGHull}. We identify the joint distribution of these two processes in a similar way as in \cite{CLGHull}. Let $Z^h$ be the subcritical continuous-state branching process with branching mechanism $$\psi(\lambda)=\sqrt{\frac{8}{3}} \lambda \sqrt{\lambda+3}, \qquad \lambda >0.$$ Moreover, for all $\delta>0$, let $\nu_{\delta}$ be the following measure on $\R^+$:
\[ \nu_{\delta}( \mathrm{d}x)= \frac{\delta^3 e^{2 \delta}}{1+2 \delta} \frac{e^{-\frac{\delta^2}{2x}-2x}}{\sqrt{2 \pi x^5}} \mathbbm{1}_{x>0} \, \mathrm{d}x. \]

\begin{thm}[Perimeter and volume processes of $\hp$] \label{identification}\
\vspace{-2mm}
\begin{itemize}
\item[1)]
The perimeter process $P^h$ of the hyperbolic Brownian plane has the same distribution as the time-reversal of $Z^h$, started from $+\infty$ at time $-\infty$, and conditioned to die at time $0$.
\vspace{-1mm}
\item[2)]
Conditionally on $P^h$, the process $V^h$ has the same distribution as the process
\[ \Big( \sum_{s_i \leq r} \xi_i^h \Big)_{r \geq 0} ,\]
where $(s_i)$ is a measurable enumeration of the jumps of $P^h$, the random variables $\xi^h_i$ are independent and $\xi^h_i$ has distribution $\nu_{|\Delta P_{s_{i}}^h|}$ for all $i$.
\end{itemize}
\end{thm}

This allows us to compute the asymptotics of these processes as in the discrete case in \cite{CurPSHIT}.

\begin{corr}[Exponential growth]\label{asymptoticspv}
We have the convergences
\[ \frac{P_r^h}{e^{2 \sqrt{2}r}}  \xrightarrow[r\to +\infty]{a.s.}  \mathcal{E} \quad \mbox{ and }\quad  \frac{V_r^h}{P_r^h}  \xrightarrow[r\to +\infty]{a.s.} \frac{1}{4}, \]
where $ \mathcal{E}$ is an exponential variable of parameter $12$.
\end{corr}

The structure of the paper is as follows. In Section 1 we introduce the type-I analog of the PSHT, and show they are the only type-I triangulations enjoying a similar domain Markov property as that defined by Curien in \cite{CurPSHIT}. We also gather a few enumeration results. Section 2 is devoted to the proof of Theorem \ref{mainthm} and \ref{joint} and Section 3 to the study of the perimeter and volume processes. Appendix A contains a technical result about the Gromov--Hausdorff--Prokhorov convergence. It shows that under some technical assumptions, if a sequence $(X_n)$ of metric spaces converges to $X$, then the hulls $\overline{B_r} (X_n)$ converge to $\overline{B_r}(X)$.

\paragraph{Acknowledgments:} I thank Nicolas Curien for suggesting me to study this object, and for carefully reading many earlier versions of this manuscript. I also thank the anonymous referee for his useful comments. I acknowledge the support of ANR Liouville (ANR-15-CE40-0013) and ANR GRAAL (ANR-14-CE25-0014).

\tableofcontents

\section{Prerequisites: enumeration and type-I PSHT}

\subsection{Combinatorial preliminaries}

A \textit{type-I triangulation of a $p$-gon} is a planar map equipped with a distinguished oriented edge called the \textit{root}, in which the face to the right of the root has a simple boundary of length $p$ and every other face has degree $3$. It may contain multiple edges and self-loops. In this whole work we will make repeated use of the results of Krikun \cite{Kri07} about the enumeration of type-I triangulations. For $p \geq 1$ and $n \geq 0$, we write $\mathscr{T}_{n,p}$ for the set of type-I triangulations of a $p$-gon with $n$ inner vertices, and $\# \mathscr{T}_{n,p}$ for its cardinal. By Euler's formula, a triangulation of a $p$-gon with $n$ inner vertices has $3n+2p-3$ edges. Hence, the main theorem of \cite{Kri07} in the case $r=0$ (that is, triangulations with only one hole) can be rewritten
\begin{equation}\label{enumeration}
\# \mathscr{T}_{n,p}=\frac{p(2p)!}{(p!)^2} \frac{4^{n-1} (2p+3n-5)!!}{n! (2p+n-1)!!} \underset{n \to +\infty}{\sim} C(p) \lambda_c^{-n} n^{-5/2},
\end{equation}
where we recall that $\lambda_c=\frac{1}{12 \sqrt{3}}$, and where
\begin{equation} \label{enumasymp}
C(p) = \frac{3^{p-2} p (2p)!}{4 \sqrt{2 \pi} (p!)^2} \underset{p \to +\infty}{\sim} \frac{1}{36 \pi \sqrt{2}} 12^p  \sqrt{p}.
\end{equation}

A triangulation of the sphere with $n$ vertices can be seen after a root transformation as a triangulation of a $1$-gon with $n-1$ inner vertices (see Figure 2 in \cite{CLGmodif}). Hence, the number of triangulations of the sphere with $n$ vertices is
\begin{equation}\label{asympsphere}
\# \mathscr{T}_{n-1,1} \underset{n \to +\infty}{\sim} \frac{1}{72 \sqrt{6\pi}} \lambda_c^{-n} n^{-5/2}.
\end{equation}
We also write $Z_p(\lambda)=\sum_{n \geq 0} \# \mathscr{T}_{n,p} \lambda^n$. Note that, by the asymptotics \eqref{enumeration}, we have $Z_p(\lambda)<+\infty$ iff $\lambda \leq \lambda_c$. We finally write $ G_{\lambda}(x)=\sum_{p \geq 1} Z_p(\lambda) x^p$. Formula (4) of \cite{Kri07} computes $G_{\lambda}$ after a simple change of variables:
\begin{equation}\label{G}
G_{\lambda}(x) = \frac{\lambda}{2} \bigg( \Big(1-\frac{1+8h}{h}x \Big) \sqrt{1-4(1+8h)x} -1+\frac{x}{\lambda}\bigg),
\end{equation}
where $h \in \big( 0, \frac{1}{4} \big]$ is such that 
\begin{equation}\label{h-lambda}
\lambda=\frac{h}{(1+8h)^{3/2}}.
\end{equation}
Note that our $h$ corresponds to the $h^3$ of \cite{Kri07}. From \eqref{G} we easily get
\begin{equation}\label{Z_1}
Z_1(\lambda)=\frac{1}{2}-\frac{1+2h}{2\sqrt{1+8h}}
\end{equation}
and, for $p \geq 2$,
\begin{equation}\label{Z_p}
Z_p(\lambda)=( 2+16h)^{p}\frac{(2p-5)!!}{p!} \frac{(1-4h)p+6h}{4 (1+8h)^{3/2} }.
\end{equation}

We now prove a combinatorial estimate that we will use later in the proof of the convergence of the type-I UIPT to the Brownian plane.

\begin{lem}\label{lemcombi1}
When $n,p \to +\infty$ with $p=O \big( \sqrt{n} \big)$, we have
\[\# \mathscr{T}_{n,p} \sim \frac{1}{36 \pi \sqrt{2}} \lambda_c^{-n} n^{-5/2} 12^p \sqrt{p} \, \exp \Big(-\frac{2p^2}{3n} \Big).\]
\end{lem}

\begin{proof}
This follows from developping \eqref{enumeration} asymptotically using the Stirling formula. The same estimate for type-II triangulations can be found in the proof of Proposition 8 of \cite{CLGpeeling}. Only some constants differ, and these constants for type-I triangulations are given in Section 6.1 of \cite{CLGpeeling}. We omit the details here.
\end{proof}

\subsection{Definition of the type-I PSHT}

The goal of this section is to construct the analog of the hyperbolic triangulations of \cite{CurPSHIT} in the case of type-I triangulations. Since the construction is roughly the same, we only stress the differences. If $t$ is a finite, rooted triangulation with a simple hole of perimeter $p$, we write $|t|$ for its number of vertices. By $t \subset T$, we mean that $T$ may be obtained by filling the hole of $t$ with an infinite triangulation of perimeter $p$.

\begin{defn}\label{defMarkov}
Let $\lambda>0$. A random (rooted) infinite type-I triangulation of the plane $T$ is $\lambda$-\emph{Markovian} if there are constants $\big( C_p(\lambda) \big)_{p \geq 1}$ such that, for all finite rooted triangulations $t$ with a hole of perimeter $p$, we have
\[\P \big( t \subset T \big)=C_p(\lambda) \lambda^{|t|}.\]
\end{defn}

\begin{rem}
Like Curien \cite{CurPSHIT}, we choose a stronger definition of the Markov property than that of Angel $\&$ Ray \cite{AR13}. Although these two definitions should coincide for type-II triangulations, this is important in the context of type-I triangulations. Indeed, a weaker definition would allow a much larger class of Markovian triangulations (see \cite{AR13}, Section 3.4 for a precise discussion in the half-planar case).
\end{rem}

\begin{prop}
If $\lambda > \lambda_c$, there is no $\lambda$-Markovian type-I triangulation.
If $\lambda \leq \lambda_c$, there is a unique one (in distribution). Besides we have
\begin{equation}\label{Cformula}
C_p(\lambda)=\frac{1}{\lambda} \Big( 8+\frac{1}{h}\Big)^{p-1} \sum_{q=0}^{p-1} \binom{2q}{q} h^{q},
\end{equation}
where $h$ is like in \eqref{h-lambda}. We will write $\mathbb{T}_{\lambda}$ for this triangulation and $\mathbb{T}=\mathbb{T}_{\lambda_c}$, which coincides with the type-I UIPT \cite{CLGmodif,St14}.
\end{prop}

\begin{proof}
The uniqueness can be proved along the same lines as in Section $1$ of \cite{CurPSHIT}. The analog of relation $(5)$ in \cite{CurPSHIT} is, for all $p \geq 1$,
\begin{equation} \label{recrel}
C_p(\lambda)=\lambda C_{p+1}(\lambda) +2 \sum_{i=0}^{p-1} C_{p-i}(\lambda) Z_{i+1}(\lambda).
\end{equation}
Note that in our case, the sum starts at $0$ and ends at $p-1$ (instead of $1$ and $p-2$ in \cite{CurPSHIT}) because of the possible presence of loops. Hence, the $\lambda \leq \lambda_c$ condition comes from the fact that the radius of convergence of $Z_p$ is $\lambda_c$ by \eqref{enumeration}. If we write $F_{\lambda}(x)=\sum_{p \geq 0} C_p(\lambda) x^p$, then \eqref{recrel} becomes 
\[F_{\lambda}(x)=\frac{\lambda}{x} \Big( F_{\lambda}(x)-C_1(\lambda)x \Big)+\frac{2}{x} G_{\lambda}(x) F_{\lambda}(x),\]
so
\begin{equation}\label{FandG}
F_{\lambda}(x)=\frac{\lambda C_1(\lambda) x}{\lambda-x+2G_{\lambda}(x)}.
\end{equation}
By combining \eqref{FandG} and \eqref{G} we get
\begin{equation} \label{F}
F_{\lambda}(x)=\frac{C_1(\lambda) x}{\Big(1-\frac{1+8h}{h}x \Big) \sqrt{1-4(1+8h)x}}.
\end{equation}
Finally, for all $p \geq 1$,
\begin{eqnarray*}
C_p(\lambda) &=& C_1(\lambda) \sum_{q+r+1=p} \Big(\frac{1+8h}{h} \Big)^r \binom{-1/2}{q} (-4)^q (1+8h)^q\\
&=& C_1(\lambda) \sum_{q=0}^{p-1} \Big(\frac{1+8h}{h} \Big)^{p-1-q} \frac{(-1)^q}{4^q} \binom{2q}{q} (-4)^q (1+8h)^q\\
&=&  C_1(\lambda) \Big(8+\frac{1}{h} \Big)^{p-1} \sum_{q=0}^{p-1} \binom{2q}{q} h^{q}.
\end{eqnarray*}

To prove the uniqueness and obtain the desired formula, it only remains to prove that we must have $C_1(\lambda)=\frac{1}{\lambda}$. Let $t_0$ be the map consisting of a single loop. Since any triangulation of the sphere can be seen as a triangulation of a $1$-gon (see Figure 2 in \cite{CLGmodif}), we must have $t_0 \subset \mathbb{T}_{\lambda}$ with probability $1$. Hence, $C_1(\lambda)=\frac{1}{\lambda}$.

The proof of the existence is essentially the same as in \cite{CurPSHIT}: consider the sequence $\big( C_p(\lambda) \big)_{p \geq 1}$ given by \eqref{Cformula} with $C_p(\lambda)>0$ for all $p$. It verifies \eqref{recrel}, so for all $p\geq 1$, we have
\[\lambda \frac{C_{p+1}(\lambda)}{C_p(\lambda)} +2 \sum_{i=0}^{p-1} Z_{i+1}(\lambda) \frac{C_{p-i}(\lambda)}{C_p(\lambda)}=1.\]
The last display can be interpreted as transition probabilities for the peeling process of $\mathbb{T}_{\lambda}$. This allows us to construct a random triangulation by peeling like in \cite{CurPSHIT}. The same arguments as in \cite{CurPSHIT} prove that we get a triangulation $\mathbb{T}_{\lambda}$ of the plane, that the distribution of $\mathbb{T}_{\lambda}$ is independent of the peeling algorithm used for the construction, and that $\mathbb{T}_{\lambda}$ is $\lambda$-Markovian.
\end{proof}

We note that in the critical case, we have a more explicit expression of $C_p(\lambda)$: we have $h=\frac{1}{4}$, so
\begin{equation}\label{critical}
C_p \big( \lambda_c \big) = \lambda_c^{-1} \times 12^{p-1} \sum_{q=0}^{p-1} \frac{1}{4^q} \binom{2q}{q}=2 \sqrt{3} \times 3^p \frac{p (2p)!}{p!^2},
\end{equation}
as easily proved by induction.

We will later need precise asymptotics of the numbers $C_p(\lambda)$. For that purpose, we already state the following estimate.

\begin{lem}\label{lemcombi2}
Let $(p_n)_{n \geq 0}$ be a sequence of positive integers such that $\frac{p_n}{n^2} \to \frac{3}{2}p$ where $p \geq 0$. Let also $(h_n)_{n \geq 0}$ be a sequence of numbers in $\big(0, \frac{1}{4} \big]$ such that
\[h_n=\frac{1}{4}-\frac{1}{2n^2}+o \Big( \frac{1}{n^2} \Big).\]
Then we have
\[ \frac{1}{\sqrt{p_n}} \sum_{q=0}^{p_n-1} \binom{2q}{q} h_n^q \xrightarrow[n \to +\infty]{}  \frac{2}{\sqrt{\pi}} \int_0^1 e^{-3px^2} \mathrm{d}x.\]
\end{lem}

\begin{proof}
Note that if $q=\lfloor x n^2 \rfloor$ with $x \in \big(0,\frac{3}{2}p \big)$, then
\[\frac{1}{\sqrt{p_n}} \binom{2q}{q}h_n^q \underset{n \to +\infty}{\sim} \frac{1}{n^2} \frac{1}{\sqrt{\pi x}} \sqrt{\frac{3}{2p}} e^{-2x}.\]
The Riemann sums are easily seen to converge to
\[ \sqrt{\frac{2}{3p}} \int_0^{3p/2} \frac{1}{\sqrt{\pi y}} e^{-2y} \mathrm{d}y,\]
which is equal to the desired integral after the change of variables $y=\frac{3p}{2} x^2$. The details are left to the reader.
\end{proof}

\section{Convergence to the hyperbolic Brownian plane}

\subsection{About the Gromov--Hausdorff--Prokhorov convergence}

We first recall from \cite{ADH13} the definition of the (bipointed) Gromov--Hausdorff--Prokhorov distance.

\begin{defn}\label{defGHP}
Let $\big( (X_1,d_1), x_1, y_1, \mu_1 \big)$ and $\big( (X_2,d_2), x_2, y_2, \mu_2 \big)$ be two compact bipointed measured metric spaces. We assume the measures $\mu_1$ and $\mu_2$ are finite (but they do not have to be probability measures). The \textit{Gromov--Hausdorff--Prokhorov distance} (we will sometimes write GHP distance and denote it by $d_{GHP}$) between $X_1$ and $X_2$ is the infimum of all $\eps>0$ for which there are isometrical embeddings $\Psi_1$ and $\Psi_2$ of $X_1$ and $X_2$ in the same metric space $(Z,d)$ such that:
\begin{itemize}
\item[a)]
$\Psi_1(x_1)=\Psi_2(x_2)$,
\item[b)]
$d \big( \Psi_1(y_1),\Psi_2(y_2) \big) \leq \eps$,
\item[c)]
the Hausdorff distance between $\Psi_1(X_1)$ and $\Psi_2(X_2)$ is not greater than $\eps$,
\item[d)]
the Lévy--Prokhorov distance between $\mu_1 \circ \Psi_1^{-1}$ and $\mu_2 \circ \Psi_2^{-1}$ is not greater than $\eps$.
\end{itemize}
The same definition holds for pointed measured compact metric spaces. We just need to withdraw condition b).
\end{defn}

If $\big( (X, d), x, \mu \big)$ is a pointed measured metric space and $r \geq 0$, we write $B_r(X)$ for the closed ball of radius $r$ centered at $x$ in $X$, equipped with the restrictions of the distance $d$ and of the measure $\mu$.

\begin{defn}
Let $\big( (X_1, d_1), x_1, \mu_1 \big)$ and $\big( (X_2, d_2), x_2, \mu_2 \big)$ be two locally compact pointed measured metric spaces. The $\textit{local Gromov--Hausdorff--Prokhorov distance}$ between $X_1$ and $X_2$, which we will denote by $d_{LGHP}(X_1, X_2)$, is the sum
\[ \sum_{r \geq 1} \frac{1}{2^r} \max \Big( 1, d_{GHP} \big( B_r(X_1), B_r(X_2) \big) \Big).\]
\end{defn}

\begin{defn}
Let $\big( (X, d), x, y, \mu \big)$ be a locally compact bi-pointed measured metric space. The \textit{hull} of center $x$ and radius $r$ with respect to $y$ is the union of the closed ball of radius $r$ centered at $x$ and all the connected components of its complementary that do not contain $y$. It is denoted by $\overline{B_r}(X,x,y)$. We will write $\overline{B_r}(X)$ when there is no ambiguity. Equipped with the restrictions of $d$ and $\mu$, it is a compact measured metric space.

If $\big( (X, d), x \big)$ is unbounded and one-ended we will omit the second distinguished point: the hull will be the union of the ball of radius $r$ centered at $x$ and all the bounded connected components of its complementary. It means that $y$ is at infinity.
\end{defn}

Recall that there are two natural ways to equip a part $A$ of $(X,d)$ with a metric: the \textit{induced} metric, i.e. the restriction of $d$ to $A$, and the \textit{intrisic} one, which makes $A$ a geodesic space when it is well-defined (see \cite{BBI01}, Chapter 2.3). In order to avoid further confusions, we insist that $\overline{B_r}(X)$ is equipped with the induced distance. If $m$ is a map, we will also write $B_r(m)$ for the map consisting of all the faces of $m$ having at least one vertex at distance at most $r-1$ from the root vertex, along with all their vertices and edges. We write $B_r^{\bullet}(m)$ for the map that is the union of $B_r(m)$ and all the bounded connected components of its complement. When $m$ is seen as a metric space, the hull $\overline{B_r}(m)$ has the same set of vertices as $B_r^{\bullet}(m)$. However, the distances inherited from $m$ are not the same as those in $B_r^{\bullet}(m)$. We will always see $\overline{B_r}(m)$ as a metric space and $B_r^{\bullet}(m)$ as a map.

We will need several times to deduce properties of one of these distances from properties of the other. To this end, we point out that if $m$ is a map, then $\overline{B_r}(m)$ is a measurable function of $B_{2r}^{\bullet}(m)$ for all $r \geq 0$. Indeed, any geodesic in $m$ between two vertices $x$ and $y$ of $\overline{B_r}(m)$ must stay in $B_{2r}^{\bullet} (m)$, so the distances between $x$ and $y$ in $m$ and in $B_{2r}^{\bullet}(m)$ coincide.

We will also need the following technical result that is proved in Appendix A.

\begin{prop}\label{GHP}
Let $\big( (X_n, d_n), x_n, \mu_n \big)$ be a sequence of unbounded, locally compact, pointed measured metric spaces. Assume that $\big( (X_n, d_n), x_n, \mu_n \big)$ converges for the local GHP distance to a measured metric space $\big( (X,d), x, \mu)$. Let $r \geq 0$. We assume that:
\begin{itemize}
\item[(i)]
$X$ and the $X_n$ are one-ended length spaces,
\item[(ii)]
every non-empty open subset of $X$ has positive measure,
\item[(iii)]
the function $V: s \to \mu \big( \overline{B_s}(X) \big)$ is continuous at $r$.
\end{itemize}

\noindent Then:
\begin{itemize}
\item[1)]
$\overline{B_r}(X_n)$ converge for the GHP distance to $\overline{B_r}(X)$,
\item[2)]
in particular, we have the convergence
\[\mu_n \big( \overline{B_r}(X_n) \big) \longrightarrow \mu \big( \overline{B_r}(X) \big).\]
\end{itemize}
Moreover, the proposition also holds for bipointed, compact spaces $\big( (X_n, d_n), x_n,y_n, \mu_n \big)$ and $\big( (X, d), x,y, \mu \big)$.
\end{prop}

\subsection{Convergence of the type-I UIPT to the Brownian plane}

If $(X,d)$ is a metric space and $\alpha>0$, we will write $\alpha X$ for the metric space $(X, \alpha d)$. We recall that $\mathbb{T}=\mathbb{T}_{\lambda_c}$ is the type-I UIPT. If $t$ is a (possibly infinite) triangulation, recall that $B_r(t)$ denotes its ball of radius $r$ around the origin of its root edge and $\overline{B_r}(t)$ its hull, endowed with the induced metric. We denote by $\p$ the Brownian plane defined in \cite{CLGplane}. Our goal in this section is to prove Theorem \ref{joint}. We start with the first two marginals, whereas the convergence of the third one will be the content of Section 2.3.

\begin{prop} \label{convergenceBP}
Let $\mu_{\mathbb{T}}$ be the measure on $\mathbb{T}$ giving mass $1$ to each vertex, and let $\mu_{\p}$ be the volume measure on $\p$ \cite{CLGplane, CLGHull}. We have the convergence
\[\Big( \frac{1}{n^{1/4}} \mathbb{T}, \frac{1}{3n}\mu_{\mathbb{T}} \Big) \xrightarrow[n \to +\infty]{(d)} \big( \p, \mu_{\p} \big)\]
for the local GHP distance.
\end{prop}

We note that this result has been proved for quadrangulations in \cite{CLGplane} for the Gromov--Hausdorff distance and in \cite{W15} for the (stronger) GHP convergence. Our main tool will be the following theorem by Curien and Le Gall. It is a refinement of the convergence of uniform type-I triangulations proved by Le Gall in \cite{LG11}.

\begin{thm}[\cite{CLGmodif}, Appendix A1, Theorem 6]\label{convergenceBM}
Let $T_n$ be a uniform type-I triangulation of the sphere with $n$ vertices and $\mu_{T_n}$ the counting measure on the set of its vertices. Let also $\mathbf{m}_{\infty}$ be the Brownian map and $\mu_{\mathbf{m}_{\infty}}$ its volume measure (\cite{LG11}). The following convergence holds for the GHP distance:
\[ \Big( \frac{1}{n^{1/4}} T_n, \frac{1}{n} \mu_{T_n} \Big) \xrightarrow[n \to +\infty]{(d)} \Big( \frac{1}{3^{1/4}} \mathbf{m}_{\infty}, \mu_{\mathbf{m}_{\infty}} \Big).\]
\end{thm}

To prove Proposition \ref{convergenceBP} we need to invert the local and the scaling limit. Hence, we need the local convergence $T_n \to \mathbb{T}$ to be "uniform in the scale", which is the point of the next lemma. It parallels Proposition $1$ of \cite{CLGplane} in the case of type-I triangulations.

\begin{prop}\label{coupling}
Let $n \geq 1$. Let also $T_n^*=(T_n, y)$ be a uniform type-I triangulation of the sphere with $n$ vertices, equipped with a uniform distinguished vertex $y$. We write $B_r^{\bullet}(T_n^*)$ for the hull of radius $r$ of $T_n$, centered at the root, with respect to $y$. Then, for all $\eps>0$, there is a constant $A>0$ such that if $n > Ar^4$, there is a coupling between $T_n^*$ and $\mathbb{T}$ in which
\[ \mathbb{P} \big( B_r^{\bullet}(T_n^*) = B_r^{\bullet}(\mathbb{T}) \big) \geq 1-\eps.\]
\end{prop}

\begin{proof}
The proof of Proposition 1 in \cite{CLGplane} uses the Schaeffer bijection between maps and trees. Although a similar bijection exists for triangulations, it is more complicated. Hence, we do the computations directly on maps instead of trees as in Section 6 of \cite{CLGmodif}.

Let $\delta>0$. We know from Section 6.1 of \cite{CLGpeeling} that $\frac{1}{r^4}|B_r^{\bullet}(\mathbb{T})|$ and $\frac{1}{r^2}|\partial B_r^{\bullet}(\mathbb{T})|$ converge in distribution to a.s. positive random variables. Hence, there are positive constants $c_{\delta}$ and $C_{\delta}$ such that, for $r$ large enough, we have
\[ \P \left( c_{\delta} \, r^2 \leq |\partial B_r^{\bullet}(\mathbb{T})| \leq C_{\delta} \, r^2 \mbox{ and } c_{\delta} \, r^4 \leq |B_r^{\bullet}(\mathbb{T})| \leq C_{\delta} \, r^4 \right) \geq 1-\delta.\]
Now take $m$ and $p$ such that $c_{\delta} \, r^2 \leq p \leq C_{\delta} r^2$ and $c_{\delta} \, r^4 \leq m \leq C_{\delta} r^4$. Let $t$ be a triangulation of a $p$-gon with $m$ vertices (including the boundary) such that $t$ is a possible value of $B_r^{\bullet}(\mathbb{T})$. On the one hand, we have
\[ \mathbb{P} \big( B_r^{\bullet}(\mathbb{T}) = t \big) \underset{\mathrm{Def \,} \ref{defMarkov}}{=} C_p ( \lambda_c ) \lambda_c^m \underset{\eqref{critical}}{=} 2 \sqrt{3} \times 3^p \frac{p (2p)!}{p!^2} \cdot \lambda_c^{m} \underset{p \to +\infty}{\sim} \frac{2\sqrt{3}}{\sqrt{\pi}} \lambda_c^m \times 12^p \sqrt{p}.\]

On the other hand, we fix $A>C_{\delta}$ and we take $n \geq Ar^4$. There are $n \# \mathscr{T}_{n-1,1}$ pointed triangulations of the sphere with $n$ vertices. Moreover, if $B_r^{\bullet}(T_n^*) = t$, there are $\# \mathscr{T}_{n-m,p}$ ways to fill the $p$-gon to complete $T_n^*$ and $n-m$ ways to choose the distinguished vertex in it (the distinguished vertex cannot lie in $B_r(T_n^*)$, since then $B_r^{\bullet}(T_n^*)$ would be the full $T_n^*$). Hence, we have $\mathbb{P} \big( B_r^{\bullet} (T_n^*) = t \big)=\frac{(n-m) \# \mathscr{T}_{n-m,p}}{n \# \mathscr{T}_{n-1,1}}$. When we let $r \to +\infty$, we have $n-m,p \to +\infty$ with $p=O \big( \sqrt{n-m} \big)$. By Lemma \ref{lemcombi1}, when $r$ goes to $+\infty$, the probability $\mathbb{P} \big( B_r^{\bullet} (T_n^*) = t \big)$ is equivalent to
\[\frac{2 \sqrt{3}}{\sqrt{\pi}} \lambda_c^m \times 12^p \sqrt{p} \exp \Big( - \frac{2p^2}{3(n-m)} \Big) \underset{r \to +\infty}{\sim} \exp \Big( - \frac{2p^2}{3(n-m)} \Big) \mathbb{P} \big( B_r^{\bullet}(\mathbb{T})=t \big) .\]
Hence, if we have chosen $A$ large enough, the following holds:
\[  (1-\delta) \mathbb{P} \big( B_r^{\bullet}(\mathbb{T})=t \big) \leq \mathbb{P} \big( B_r^{\bullet}(T_n^*) = t \big) \leq (1+\delta) \mathbb{P} \big( B_r^{\bullet}(T)=t \big), \]
as soon as $c_{\delta} r^2 \leq |\partial t| \leq C_{\delta} r^2$ and $c_{\delta} r^4 \leq |t| \leq C_{\delta} r^4$. But we know that $ B_r^{\bullet}(\mathbb{T})$ satisfies these assumptions with probability at least $1-\delta$. Hence, we can easily prove that, for $n \geq Ar^4$ and any set $B$ of finite maps, we have
\[ \big| \mathbb{P} \big( B_r^{\bullet}(\mathbb{T}) \in B \big) - \mathbb{P} \big( B_r^{\bullet}(T_n^*) \in B \big)\big| \leq 4 \delta.\]
This shows that, for $r$ large enough and $n \geq Ar^4$, the total variation distance between the distributions of $B_r^{\bullet}(\mathbb{T})$ and $B_r^{\bullet}(T_n^*)$ is less than $4\delta$, which proves the proposition.
\end{proof}

\begin{proof}[Proof of Proposition \ref{convergenceBP}]
We use Proposition \ref{coupling} with $2r$ instead of $r$. The metric spaces $B_r(T_n^*)$ and $B_r(\mathbb{T})$ (equipped with the induced distance) are measurable functions of respectively $B_{2r}^{\bullet}(T_n^*)$ and $B_{2r}^{\bullet}(\mathbb{T})$. Hence, Proposition \ref{coupling} still holds if we replace the maps $B_r^{\bullet}$ by the metric spaces $B_r$. The proof is now the same as the proof of Theorem 2 in \cite{CLGplane}, with two small modifications:
\begin{itemize}
\item
we deal with Gromov--Hausdorff--Prokhorov convergence and not only Gromov--Hausdorff, but this does not change anything in the details of the proof, see \cite{W15} for details,
\item
the constant factors are not the same: because of the factor $\frac{1}{3^{1/4}}$ in Theorem \ref{convergenceBM}, the measured spaces $\Big( \frac{1}{n^{1/4}} \mathbb{T}, \frac{1}{n}\mu_{\mathbb{T}} \Big)$ converge to $\big( \frac{1}{3^{1/4}}\p, \mu_{\p} \big)$. This has the same distribution as $\big( \p, 3 \mu_{\p} \big)$ by the scaling property of the Brownian plane.
\end{itemize}
\end{proof}

We can now prove the joint convergence of the first two marginals in Theorem \ref{joint}.

\begin{prop} \label{joint1}
We have the joint convergence
\[\Bigg( \frac{1}{n} \mathbb{T}, \Big( \frac{1}{n^4} |\overline{B_{rn}}(\mathbb{T})| \Big)_{r \geq 0} \Bigg) \xrightarrow[n \to +\infty]{(d)} \Bigg( \p, \Big( 3|\overline{B_r}(\p)| \Big)_{r \geq 0} \Bigg),\]
where the convergence of the first marginal is for the local GHP distance, and the second one for the Skorokhod topology.
\end{prop}

We will deduce Proposition \ref{joint1} from Proposition \ref{convergenceBP} thanks to the second point of Proposition \ref{GHP}. Let us check this carefully.

\begin{proof}[Proof of Proposition \ref{joint1}]
By the Skorokhod representation theorem, we may assume the convergence in Proposition \ref{convergenceBP} is almost sure. Theorem 1.4 of \cite{CLGHull} computes $\mathbb{E} \big[ e^{-|\overline{B_r}( \p)|}\big]$. In particular, it is a continuous function of $r$. Since the process $\big( |\overline{B_r}( \p)| \big)_{r \geq 0}$ has only positive jumps, it means that for all $r \geq 0$, it is almost surely continuous at $r$. Finally, the Brownian plane is defined in \cite{CLGplane} as a quotient of $\R$. This means that there is a continuous surjection from $\R$ to $\p$, and the volume measure on $\p$ is the push-forward of the Lebesgue measure under this surjection. The inverse-image of a non-empty open subset of $\p$ is a non-empty open subset of $\R$. Hence, it has positive measure, which means that any non-empty open subset of $\p$ has positive measure.

Instead of $\mathbb{T}$, we can consider the metric space $\mathbb{T}^e$, which is the union of all the vertices and edges of $\mathbb{T}$. It is equipped with the metric that makes it a geodesic space in which all edges have length $1$. We also equip this space with the counting measure $\mu_{\mathbb{T}}$ on the set of vertices. We have $d_{GHP}(B_r(\mathbb{T}),B_r(\mathbb{T}^e)) \leq 1$ for all $r$, so $d_{LGHP} \Big( \Big( \frac{1}{n^{1/4}} \mathbb{T}, \frac{1}{3n}\mu_{\mathbb{T}} \Big), \Big( \frac{1}{n^{1/4}} \mathbb{T}^e, \frac{1}{3n}\mu_{\mathbb{T}} \Big)\Big) \leq \frac{2}{n^{1/4}}$. Hence, Proposition \ref{convergenceBP} still holds if we replace $\mathbb{T}$ by $\mathbb{T}^e$. The sequence $\Big( \frac{1}{n}\mathbb{T}^e, \frac{1}{n^4} \mu_{\mathbb{T}} \Big)_{n \geq 0}$ satisfies the assumptions of Proposition \ref{GHP}, so for all $(r_1, \dots,r_k)\in \left( \R^+ \right)^k$ we have
\[ \forall i \in [\![1,k]\!], \, \frac{1}{3n^4} \mu_{\mathbb{T}} \big( \overline{B_{r_i n}}(\mathbb{T}) \big) \xrightarrow[n \to +\infty]{a.s.} \mu_{\p} \big( \overline{B_{r_i}}(\p) \big).\]
This gives the joint convergence of $\frac{1}{n} \mathbb{T}$ and of the finite-dimensional marginals of the process of volumes.

Hence, to complete the proof of Proposition \ref{joint1}, we only need tightness. The tightness of the first marginal is given by Proposition \ref{convergenceBP}. On the other hand, Theorem 2 of \cite{CLGpeeling} shows that the volume process converges. In particular, it is tight. This concludes the proof.
\end{proof}

\subsection{Joint convergence of the perimeter process}

The goal of this subsection is to prove the joint convergence of the last marginal in Theorem \ref{joint}.

\begin{proof}[Proof of Theorem \ref{joint}]
By Proposition \ref{joint1}, the first two marginal converge in distribution to $\Big( \p, \big( 3 V_r \big)_{r \geq 0} \Big)$, where $V_r=|\overline{B_r} (\p)|$.  We also know by Theorem $2$ of \cite{CLGpeeling} that $\Big( \Big( \frac{1}{n^2} |\partial \overline{B_{rn}}(\mathbb{T})| \Big)_{r \geq 0} \Big)_{n \geq 0}$ converges, so it is tight, so the triplet is tight. Hence, it is enough to prove uniqueness of the limit. Let $(n_k)$ be a subsequence along which it converges in distribution to a triplet
\[\Big( \p, \big( 3 V_r \big)_{r \geq 0}, \Big( \frac{3}{2} \widetilde{P}_r \Big)_{r \geq 0} \Big),\]
where $\widetilde{P}=\big( \widetilde{P}_r \big)_{r \geq 0}$ is a càdlàg process. We also write $P_r=|\partial \overline{B_r} (\p)|$, and we want to show $\widetilde{P}=P$.

On the one hand, Theorem 1.3 of \cite{CLGHull} describes the joint distribution of $P$ and $V$. There is a measurable enumeration $(s_i)$ of the jumps of $P$ and an i.i.d. sequence $(\xi_i)$ of variables with distribution $\nu(\mathrm{d}x)=\frac{e^{-1/2x}}{\sqrt{2 \pi x^5}} \mathbbm{1}_{x>0}\,\mathrm{d}x$ such that, for all $r \geq 0$, we have
\[V_r= \sum_{s_i \leq r} (\Delta P_{s_i})^2 \xi_i.\]
On the other hand, Theorem 2 of \cite{CLGpeeling} shows that the second and third marginals of Theorem \ref{joint} converge, and identifies the limit. Hence, it gives the distribution of the couple $(\widetilde{P}_r, V_r)_{r \geq 0}$ (see Section $6.1$ of \cite{CLGpeeling} for the computation of the constants for type-I triangulations). We get
\[(\widetilde{P}_r, V_r)_{r \geq 0}\overset{(d)}{=} (P_r, V_r)_{r \geq 0}.\]

To prove that $\widetilde{P}_r=P_r$, we show that it is possible to "track back" $P_r$ from the process $(V_r)_{r \geq 0}$, which is done in the following lemma. We will need the following notation: for a nondecreasing function $f$ and $a,b,h \in \R^+$ with $a \leq b$, we write $N_a^b(f, h)$ for the number of jumps of $f$ in $[a,b]$ of height at least $h$.

\begin{lem}\label{eta}
For all $r \geq 0$, we have
\[ \lim_{\delta \to 0} \delta^{-1} \lim_{\eps \to 0} \eps^{3/4} N_r^{r+\delta} (V, \eps) = c P_r\]
almost surely, where $c=\frac{2^{1/4}}{\pi \sqrt{3}} \Gamma \big( \frac{3}{4} \big)$.
\end{lem}

Once this lemma is proved, Theorem \ref{joint} follows easily. Indeed, since $(\widetilde{P}, V) \overset{(d)}{=} (P, V)$, for any $r \geq 0$ the variables $\widetilde{P}_r$ and $P_r$ are both the a.s. limits of the same quantity. Hence, almost surely, $\widetilde{P}_r=P_r$ for every $r \in \mathbb{Q}^+$. Since $\widetilde{P}$ and $P$ are càdlàg we have $\widetilde{P}=P$ a.s. Hence, the sequence of triplets has only one subsequential limit, which proves the theorem.
\end{proof}

\begin{rem}
Note that Proposition 1.1 of \cite{CLGHull} provides another way to "read" $P$ on the measured metric space $\p$. However, it involves the volumes of the balls $B_{r+\eps}$, which are not given by the process $V$. Our lemma is also quite similar to the main theorem of \cite{cactus2}, although much easier to prove.
\end{rem}

\begin{proof}[Proof of Lemma \ref{eta}]
Let $S^+$ be the stable spectrally negative Lévy process of index $\frac{3}{2}$ conditioned to stay positive. We normalize it in such a way that its characteristic exponent is $\psi(\lambda)=\sqrt{8/3}\,\lambda^{3/2}$. Let also $(t_i)$ be a measurable enumeration of the jumps of $S^+$, and let $(\xi'_i)$ be a sequence of i.i.d. random variables with distribution $\nu$. We write
\[Q^+_r=\sum_{t_i \leq r} \xi'_i (\Delta S^+_{t_i})^2.\]

We first show that almost surely, for any $0<a<b$, we have
\begin{equation}\label{qplus}
\lim_{\eps \to 0} \eps^{3/4} N_a^b(Q^+, \eps)=c(b-a).
\end{equation}
By monotonicity, it is enough to prove it for $a,b \in \mathbb{Q}_+^*$. Hence, it is enough to prove it for fixed $a$ and $b$. Let $Q$ be the same process as $Q^+$ but constructed from a non-conditioned stable Lévy process $S$ instead of $S^+$. Then $Q$ is a subordinator whose Lévy measure $\sigma$ is the image of $\mu \otimes \nu$ under $(x,y) \to x^2 y$, where $\mu$ is the Lévy measure of $S$. An easy computation shows that $\sigma( [ \eps, +\infty [)=c \eps^{-3/4}$ for all $\eps$. Hence, equation \eqref{qplus} for $Q$ instead of $Q^+$ follows from a law of large numbers. But since $S^+$ is absolutely continuous with respect to $S$ on $[a,b]$, equation \eqref{qplus} also holds for $Q^+$.

We now recall the law of $(P, V)$ as described in Section 4.4 of \cite{CLGpeeling}. It has the same distribution as $(S^+_{\tau_r}, Q^+_{ \tau_r})_{r \geq 0}$, where $\tau_r=\inf \{ s \geq 0 | \int_0^s \frac{\mathrm{d}u}{S^+_u}\geq r \}$ for every $r \geq 0$. Hence, we have $N_r^{r+\delta} (V,  \eps)=N_{\tau_r}^{\tau_{r+\delta}} (Q^+, \eps)$. By \eqref{qplus}, for any $r$ and $\delta$ we have (since a.s. \eqref{qplus} holds for any $0<a<b$, we can take $a$ and $b$ random)
\[\lim_{\eps \to 0} \eps^{3/4} N_{\tau_r}^{\tau_{r+\delta}} (Q^+, \eps)=c(\tau_{r+\delta}-\tau_r).\]
Now, it is easy to see, by the right-continuity of $S^+$ at $\tau_r$, that $\delta^{-1} (\tau_{r+\delta}-\tau_r) \xrightarrow[\delta \to 0]{a.s.} S^+_{\tau_r}=P_r$. This finishes the proof.
\end{proof}

\subsection{The hyperbolic Brownian plane}

The goal of this section is to prove Theorem \ref{mainthm}. We will write
\[\varphi(p, v)=e^{-2v}e^p \int_0^1 e^{-3px^2}\mathrm{d}x\]
in the whole proof of the theorem. Moreover, if $t$ is a triangulation with a simple hole of perimeter $p'$ and with $v'$ vertices in total and $\lambda \in (0,\lambda_c]$, we write
\begin{equation} \label{philambdaprime}
\varphi_{\lambda}(p',v')=\frac{\P \left( t \subset \mathbb{T}_{\lambda_n} \right)}{\P \left( t \subset \mathbb{T} \right)}.
\end{equation}
Since $\mathbb{T}_{\lambda_n}$ and $\mathbb{T}$ are Markovian, this only depends on $p'$ and $v'$ and not on $t$.

\begin{prop}\label{density}
Let $(\lambda_n)$ be a sequence of numbers in $\big( 0, \lambda_c \big]$ that satisfies \eqref{lambdaspeed}. Let $r>0$ and let $(p'_n)$ and $(v'_n)$ be two sequences of positive integers such that $\frac{p'_n}{n^2} \to \frac{3}{2}p$ and $\frac{v'_n}{n^4} \to 3 v$. Then
\[\varphi_{\lambda_n}(p'_n, v'_n) \xrightarrow[n \to +\infty]{} \varphi (p,v).\]
\end{prop}

\begin{proof}
This is just a matter of gathering our estimates in Section 1.1 and 1.2 together. Let us proceed: for every $n$, let $h_n \in \big( 0, \frac{1}{4} \big]$ be such that $\lambda_n=\frac{h_n}{(1+8h_n)^{3/2}}$. It is easy to check that
\begin{equation}\label{hn}
h_n=\frac{1}{4}-\frac{1}{2 n^2}+o \Big( \frac{1}{n^2} \Big).
\end{equation}

We know from \eqref{Cformula} that
\begin{eqnarray*}
\varphi_{\lambda_n}(p'_n,v'_n) &\underset{\mathrm{Def. \,} \ref{defMarkov}}{=}& \frac{\lambda_n^{v_n} C_{p_n}(\lambda_n)} {\lambda_c^{v_n} C_{p_n}(\lambda_c)}\\
&\underset{\eqref{critical}}{=}& \Big( \frac{\lambda_n}{\lambda_c} \Big)^{v_n} \frac{p_n!^2}{2 \sqrt{3} \times 3^{p_n} p_n (2p_n)!} \frac{1}{\lambda_n} \Big(8+\frac{1}{h_n} \Big)^{p_n-1} \sum_{q=0}^{p_n-1} \binom{2q}{q} h_n^{q}\\
&=& \Big( \frac{\lambda_n}{\lambda_c} \Big)^{3v n^4+o(n^4)} \Big(\frac{2}{3}+\frac{1}{12h_n} \Big)^{p_n-1} \frac{\lambda_c}{\lambda_n} \frac{4^{p_n} p_n !^2}{2 p_n (2p_n)!} \sum_{q=0}^{p_n-1} \binom{2q}{q} h_n^{q}.
\end{eqnarray*}
The first factor converges to $e^{-2v}$ because $\frac{\lambda_n}{\lambda_c}=1-\frac{2}{3n^4}+o \Big( \frac{1}{n^4} \Big)$. The second factor is also easy to estimate:
\[\Big(\frac{2}{3}+\frac{1}{12h_n} \Big)^{p_n-1}=\Big( 1+\frac{2}{3n^2}+o \Big( \frac{1}{n^2} \Big) \Big)^{\frac{3p}{2}n^2+o(n^2)} \xrightarrow[n \to +\infty]{} e^{p}.\]
The third one converges to one. By the Stirling formula, we have $\frac{4^{p_n}  p_n!^2}{2p_n (2p_n)!} \sim \frac{\sqrt{\pi}}{2 \sqrt{p_n}}$. Finally, Lemma \ref{lemcombi2} gives an asymptotic equivalent of the last factor and we are done.
\end{proof}

The last proposition is more or less equivalent to vague convergence. We now need to show that no mass "escapes" at infinity, i.e. that the total mass of the limit measure is $1$.

\begin{lem}\label{mass}
Recall that $P_r=|\partial \overline{B_r}(\p)|$ and $V_r=|\overline{B_r}(\p)|$. For every $r \geq 0$, we have
\[\E \big[ \varphi(V_r, P_r) \big]=1.\]
\end{lem}

\begin{proof}
We use the expression of the Laplace transform of $(P_r,V_r)$ that is computed in \cite{CLGHull}. First, Theorem 1.4 of \cite{CLGHull} computes $\E \left[ e^{-\mu V_r} | P_r=\ell \right]$ for $\mu, \ell \geq 0$. We apply it with $\mu=2$:
\[\E \Big[ \exp ( -2V_r ) \Big|  P_r=\ell\Big]=2\sqrt{2} r^3 \frac{\cosh (\sqrt{2}r )}{\sinh^{3} (\sqrt{2}r)} \exp \Big(-\ell \Big(3 \coth^{2} (\sqrt{2}r)-2-\frac{3}{2r^2} \Big) \Big).\]

We now apply Fubini's theorem, and use the Laplace transform of $P_r$ given by Theorem 1.2 of \cite{CLGHull}:
\begin{eqnarray*}
\E \big[ \varphi(V_r, P_r) \big] &=& \E \Big[ \E \big[ \exp ( -2V_r) \big|  P_r \big] \exp \big( P_r \big) \int_0^1 \exp \big( -3P_r x^2\big) \mathrm{d}x\Big]\\
&=& 2\sqrt{2} r^3 \frac{\cosh (\sqrt{2}r )}{\sinh^{3} (\sqrt{2}r)} \int_0^1 \E \Bigg[ \exp \Big(-P_r \Big(3 \coth^{2} (\sqrt{2}r)-2-\frac{3}{2r^2} \Big) \Big) \exp \Big( (1-3x^2)P_r \Big) \Bigg] \mathrm{d}x \\
&=& 2\sqrt{2} r^3 \frac{\cosh (\sqrt{2}r )}{\sinh^{3} (\sqrt{2}r)} \int_0^1 \Bigg(1+\frac{2r^2}{3} \Big(3 \coth^{2} (\sqrt{2}r)-2-\frac{3}{2r^2}-1+3x^2 \Big) \Bigg)^{-3/2} \mathrm{d}x\\
&=& 2\sqrt{2} r^3 \frac{\cosh (\sqrt{2}r )}{\sinh^{3} (\sqrt{2}r)} \int_0^1 \Bigg(2r^2 \coth^{2} (\sqrt{2}r) -2r^2+2r^2 x^2  \Bigg)^{-3/2} \mathrm{d}x\\
&=& \frac{\cosh (\sqrt{2}r )}{\sinh^{3} (\sqrt{2}r)} \int_0^1 \Big( \frac{1}{\sinh^2 (\sqrt{2}r)} +x^2 \Big)^{-3/2} \mathrm{d}x\\
&=& \cosh (\sqrt{2}r) \int_0^1 \Big(1+\sinh^2 (\sqrt{2} r)x^2 \Big)^{-3/2} \mathrm{d}x\\
&=& \cosh (\sqrt{2}r) \frac{1}{\sqrt{1+\sinh^2 (\sqrt{2} r)}}\\
&=& 1,
\end{eqnarray*}
where at the end we used the fact that $(1+ax^2)^{-3/2}$ is the derivative of $\frac{x}{\sqrt{1+ax^2}}$.
\end{proof}

Our main theorem is now easy to prove.

\begin{proof}[Proof of Theorem \ref{mainthm}]
In this proof, if $X$ is a measured metric space, we will write $\frac{1}{n} X$ for the metric space obtained from $X$ by multiplying all distances by $\frac{1}{n}$ and the measure by $\frac{1}{3n^4}$.

Let $r > 0$ and let $\mathbb{T}$ be a type-I UIPT. Almost surely, the processes $P$ and $V$ have no jump at $r$ and at $2r$, so we have convergence of the one-dimensional marginal at $2r$ in Theorem \ref{joint}. For all $n \geq 1$, let $P'_n=|\partial B_{2rn}^{\bullet}(\mathbb{T})|$ and $V'_n=| B_{2rn}^{\bullet}(\mathbb{T})|$. By the Skorokhod representation theorem, we may assume, as $n \to +\infty$, the convergences
\begin{equation} \label{triple_convergence}
\left \{
\begin{array}{lcl}
 \frac{1}{n} \mathbb{T} &\xrightarrow[LGHP]{a.s.}& \p\\
\frac{1}{n^2} P'_n &\xrightarrow{a.s.}& \frac{3}{2} P_{2r}\\
\frac{1}{n^4} V'_n &\xrightarrow{a.s.}& 3 V_{2r}.
\end{array}
\right.
\end{equation}
We have already verified in the proof of Theorem \ref{joint1} that the assumptions of Proposition \ref{GHP} are satisfied. By the first point of Proposition \ref{GHP}, the convergence of $\frac{1}{n} \mathbb{T}$ to $\p$ implies
\begin{equation} \label{convergence_hulls}
\frac{1}{n} \overline{B_{rn}}(\mathbb{T}) \xrightarrow[GHP]{a.s.} \overline{B_r}(\p).
\end{equation}

Let  $\mathscr{G}$ denote the GHP space and let $f$ be a bounded, continuous function from $\mathscr{G}$ to $\R^+$. We recall that the space $\overline{B_{rn}}(\mathbb{T})$ is a measurable function of the map $B_{2rn}^{\bullet}(\mathbb{T})$. Hence, there is a function $\widetilde{f}$ (that depends on $n$) from the set of finite triangulations to $\mathbb{R}^+$ such that $f \Big( \frac{1}{n} \overline{B_{rn}} (t) \Big)=\widetilde{f} \big( B_{2rn}^{\bullet} (t) \big)$ for any one-ended, infinite triangulation $t$. Hence, by the definition \eqref{philambdaprime} of $\varphi_{\lambda}$, we can write
\begin{eqnarray*}
\E \Big[ f \Big( \frac{1}{n} \overline{B_{rn}} (\mathbb{T}_{\lambda_n}) \Big) \Big] &=& \E \Big[ \widetilde{f} \big( B_{2rn}^{\bullet} (\mathbb{T}_{\lambda_n}) \big) \Big]\\
&=& \E \Big[ \varphi_{\lambda_n}(P'_n, V'_n) \, \widetilde{f} \big( B_{2rn}^{\bullet} (\mathbb{T}) \big) \Big]\\
&=& \E \Big[ \varphi_{\lambda_n}(P'_n, V'_n) \, f \Big( \frac{1}{n} \overline{B_{rn}}(\mathbb{T}) \Big) \Big].
\end{eqnarray*}
By the convergences \eqref{triple_convergence} and \eqref{convergence_hulls}, Proposition \ref{density} and the continuity of $f$, the expression inside the expectation converges a.s. to $\varphi \big(P_{2r}, V_{2r} \big) f \big( \overline{B_r}(\p) \big)$. By Fatou, we have
\begin{equation}\label{fatou}
\liminf_n \E \Big[ f \Big( \frac{1}{n} \overline{B_{rn}} (\mathbb{T}_{\lambda_n}) \Big) \Big] \geq \E \Big[ \varphi \big( P_{2r}, V_{2r} \big) f \big( \overline{B_r}(\p) \big) \Big].
\end{equation}
Let $M \in \R^+$ be such that $f \leq M$. By Lemma \ref{mass} we have
\[\E \Big[ M \varphi \big(P_{2r}, V_{2r} \big) \Big]=M,\]
so by applying \eqref{fatou} to $M-f$ we get the reverse inequality. This shows that $\frac{1}{n} \overline{B_{rn}}(\mathbb{T}_{\lambda_n})$ converges in distribution to the random metric space having density $\varphi \big(P_{2r}, V_{2r} \big)$ with respect to $\overline{B_r}(\p)$. We denote this space by $\overline{B_r}^h$.

Moreover, let $0<s \leq r$. We can take points $y_n$ on the boundary of $\frac{1}{n} \overline{B_{rn}}(\mathbb{T}_{\lambda_n})$ and $y \in \overline{B_r}^h$ such that $\Big( \frac{1}{n} \overline{B_{rn}}(\mathbb{T}_{\lambda_n}), y_n \Big)$ converges in distribution to $\big( \overline{B_r}^h, y \big)$. By Proposition \ref{GHP}, we have
\[\overline{B_s} \Big( \frac{1}{n} \overline{B_{rn}}(\mathbb{T}_{\lambda_n}) \Big) \xrightarrow[n \to +\infty]{(d)} \overline{B_s} \big( \overline{B_r}^h \big),\]
where the hulls are taken with respect to $y_n$ and $y$ respectively. More precisely, we use here Proposition \ref{GHP} in the compact, bipointed case. But $\overline{B_s} \Big( \frac{1}{n} \overline{B_{rn}}(\mathbb{T}_{\lambda_n}) \Big) = \frac{1}{n} \overline{B_{sn}}(\mathbb{T}_{\lambda_n})$, which converges in distribution to $\overline{B_s}^h$. Hence, we have
\[\overline{B_s} \big( \overline{B_r}^h \big) \overset{(d)}{=} \overline{B_s}^h,\]
i.e. the $\overline{B_r}^h$ are "consistent". By the Kolmogorov extension theorem, there is a random metric space, that we write $\hp$, such that $\overline{B_r}(\hp) \overset{(d)}{=} \overline{B_r}^h$ for all $r\geq 0$, and we have
\[\frac{1}{n} \mathbb{T}_{\lambda_n}  \xrightarrow[n \to +\infty]{(d)} \hp\]
for the local GHP distance.
\end{proof}

\section{Properties of the hyperbolic Brownian plane}

\subsection{Local properties of the hyperbolic Brownian plane}

The absolute continuity relation between the Brownian plane $\p$ and the hyperbolic Brownian plane $\hp$ implies that they have the same almost sure "local" properties. This gives us the two following properties of $\hp$.

\begin{prop}
Almost surely, $\hp$ has Hausdorff dimension $4$ and is homeomorphic to $\R^2$.
\end{prop}

\begin{proof}
First, by the absolute continuity relation with the Brownian plane, for all $r \geq 0$, the space $\overline{B_r}(\hp)$ has a.s. Hausdorff dimension $4$ and so does $\hp$.

If $(X,d,\rho)$ is a pointed metric space and $r>0$, we write $U_r(X)=\bigcup_{\eps>0} \overline{B_{r-\eps}}(X)$. Let $r>0$. The set $U_r(\p)$ is a connected, open subset of $\p$, and it is quite easy to prove that $\p \backslash U_r(\p)$ is connected. Indeed, $\overline{B_{r+1}}(\p) \backslash U_r(\p)$ is connected because it is the decreasing intersection of the $\overline{B_{r+1}}(\p) \backslash \overline{B_{r-\eps}}(\p)$, which are compact, connected subsets.

The set $U_r(\p)$ is a connected, open subset of the plane whose complement is connected, so it is homeomorphic to the open unit disk (this is a consequence of the Riemann mapping theorem). In particular, almost surely, the two following hold:
\begin{enumerate}
\item
for all $x \in U_r(\p)$, there is a neighbourhood of $x$ that is homeomorphic to the unit disk,
\item
any loop in $U_r(\p)$ is contractible in $U_r(\p)$.
\end{enumerate}
By absolute continuity and since $U_r$ is always a subset of $\overline{B_r}$, the two above items a.s. hold for $U_r(\hp)$. Almost surely, they hold for any $r \in \N^*$. Hence, $\hp$ is a noncompact, simply connected topological surface. Therefore, it is homeomorphic to the plane (for example, it is a consequence of the Riemann uniformization theorem and the fact that any topological surface may be equipped with a Riemann surface structure).
\end{proof}

Note that in this proof, it was important to consider $U_r(\p)$ and not $\overline{B_r}(\p)$. Indeed, a metric space may not be homeomorphic to the plane, even if all its hulls are homeomorphic to the closed unit disk (for exemple a closed half-plane, rooted at an interior point).

For the same reason as above, the results about the local confluence of geodesics in the Brownian map and the Brownian plane also hold for the hyperbolic Brownian plane.

\begin{prop}
A.s., for any $\eps>0$, there is a $\delta>0$ such that the following holds. All the geodesics in $\hp$ from the root to a point at distance at least $\eps$ from the root share a common initial segment of length at least $\delta$.
\end{prop}

Finally, the Brownian map, the Brownian plane and the hyperbolic Brownian plane are locally isometric in the following sense.

\begin{prop}
For any $\eps>0$, there is a $\delta>0$ such that the following holds. It is possible to couple the Brownian map, the Brownian plane and the hyperbolic Brownian plane in such a way that their hulls of radius $\delta$ coincide with probability at least $1-\eps$.
\end{prop}

\begin{proof}
The result for the Brownian map and the Brownian plane is Theorem 1 of \cite{CLGplane}, so we only need to prove it for $\p$ and $\hp$. Let $\delta>0$, and let $A$ be a mesurable subset of the Gromov--Hausdorff--Prokhorov space. Then
\begin{eqnarray*}
\Big| \mathbb{P} \big( \overline{B_{\delta}} (\hp) \in A \big) - \mathbb{P} \big( \overline{B_{\delta}}(\p) \in A \big) \Big| &=& \big| \mathbb{E} \big[ \varphi(P_{\delta}, V_{\delta}) \mathbbm{1}_{\overline{B_{\delta}} (\p) \in A} \big]  - \mathbb{P} \big( \overline{B_{\delta}}(\p) \in A \big) \Big|\\
& \leq & \mathbb{E} \big[ |\varphi(P_{\delta}, V_{\delta})-1| \big],
\end{eqnarray*}
which goes to $0$ as $\delta \to 0$. Hence, the total variation distance between the distributions of $\overline{B_{\delta}} (\hp)$ and $\overline{B_{\delta}} (\p)$ goes to $0$ as $\delta \to 0$. This proves the result by the maximal coupling theorem (see e.g. Section 2 of \cite{dH12}).
\end{proof}

\subsection{Unimodularity}

The goal of this section is to prove that $\hp$ satisfies a property that is the continuum analog of unimodularity for random graphs.

\begin{prop}\label{unimodularity}
Let $\rho$ be the root of $\hp$ and $\mu_{\hp}$ its volume measure. Let $f$ be a measurable function from the space of locally compact, bipointed measured metric spaces to $\R^+$. We have
\[ \mathbb{E} \Big[ \int_{\hp} f(\hp, \rho, y) \, \mu_{\hp}(\mathrm{d}y) \Big] = \mathbb{E} \Big[ \int_{\hp} f(\hp, y, \rho) \, \mu_{\hp}(\mathrm{d} y) \Big]. \]
\end{prop}

To prove this result, we will need the following lemma. It roughly means that the degree of the root is independent of the large-scale geometry of the map.

\begin{lem}\label{rootandLGHP}
As earlier, we write $\frac{1}{n} \mathbb{T}_{\lambda_n}$ for the set of the vertices of $\mathbb{T}_{\lambda_n}$, equipped with $\frac{1}{n}$ times its graph distance and the measure giving mass $\frac{1}{3n^4}$ to each vertex. Let $\rho_n$ be the origin of the root edge of $\mathbb{T}_{\lambda_n}$. We have the convergence
\[\Big( \deg (\rho_n), \frac{1}{n} \mathbb{T}_{\lambda_n} \Big) \xrightarrow[n \to +\infty]{(d)} \big( D , \hp \big),\]
where $D$ has the same distribution as the degree of the root in the UIPT and is independent of $\hp$.
\end{lem}

\begin{proof}
The convergence of the first marginal is obvious. Hence, it is enough to prove that for all $d \geq 1$, the spaces $\frac{1}{n} \mathbb{T}_{\lambda_n}$ conditioned on $\deg (\rho_n)=d$ converge in distribution to $\hp$. More generally, we prove that for each finite triangulation $t$ that is a possible value of $B^{\bullet}_1(\mathbb{T})$, the spaces $\frac{1}{n} \mathbb{T}_{\lambda_n}$ conditioned on $B^{\bullet}_1(\mathbb{T}_{\lambda_n})=t$ converge to $\hp$. For all $s \geq 0$, the GHP distance between $B_s(\mathbb{T}_{\lambda_n})$ and $B_s(\mathbb{T}_{\lambda_n}) \backslash B_1^{\bullet}(\mathbb{T}_{\lambda_n})$ is tight (it is bounded by $|B_1^{\bullet}(\mathbb{T}_{\lambda_n})|$). Hence, for all $r \geq 0$, the GHP distance between $B_r \Big( \frac{1}{n} \mathbb{T}_{\lambda_n} \Big)$ and $B_r \Big( \frac{1}{n}  \big( \mathbb{T}_{\lambda_n} \backslash B_1^{\bullet}(\mathbb{T}_{\lambda_n}) \big) \Big)$ goes to $0$ in probability as $n \to +\infty$. Hence, it is enough to prove the lemma for $\frac{1}{n}  \big( \mathbb{T}_{\lambda_n} \backslash B_1^{\bullet}(\mathbb{T}_{\lambda_n}) \big)$ instead of $\frac{1}{n} \mathbb{T}_{\lambda_n}$. But by the spatial Markov property of $\mathbb{T}_{\lambda_n}$, the distribution of $\mathbb{T}_{\lambda_n} \backslash B_1^{\bullet}(\mathbb{T}_{\lambda_n})$ conditionally on $B_1^{\bullet}(\mathbb{T}_{\lambda_n})$ only depends on $| \partial B_1^{\bullet}(\mathbb{T}_{\lambda_n})|$.

We recall that a \emph{peeling algorithm} $\mathscr{A}$ is a way to assign, to every finite triangulation $t$ with a hole, an edge $\mathscr{A}(t)$ on the boundary of the hole. Informally, $\mathscr{A}(t)$ is the next edge to be explored once the explored part of the map is equal to $t$. See \cite{Ang03} or \cite{CLGpeeling} for more details. We now fix a deterministic peeling algorithm, and we explore $\mathbb{T}_{\lambda_n}$ using this algorithm. For every $p \geq 1$, we write $\tau_p$ for the first time at which the perimeter of the discovered map is equal to $p$. Note that, since $\mathbb{T}_{\lambda_n}$ may be identified with an infinite triangulation of a $1$-gon, the times $\tau_p$ are a.s. finite even for $p=1$ or $p=2$. We also write $D_p(\mathbb{T}_{\lambda_n})$ for the triangulation discovered at time $\tau_p$. By the spatial Markov property, the map $\mathbb{T}_{\lambda_n} \backslash B_1^{\bullet}(\mathbb{T}_{\lambda_n})$ conditionally on $| \partial B_1^{\bullet}(\mathbb{T}_{\lambda_n})|=p$ has the same distribution as $\mathbb{T}_{\lambda_n} \backslash D_p(\mathbb{T}_{\lambda_n})$. Hence, it is enough to prove that for any $p \geq 0$, we have
\[ \frac{1}{n} \big( \mathbb{T}_{\lambda_n} \backslash D_p(\mathbb{T}_{\lambda_n}) \big) \xrightarrow[n \to +\infty]{(d)} \hp. \]
This easily follows from Theorem \ref{mainthm} and the fact that, for all $r \geq 0$, the GHP distance between $B_r(\mathbb{T}_{\lambda_n})$ and $B_r(\mathbb{T}_{\lambda_n}) \backslash D_p(\mathbb{T}_{\lambda_n})$ is at most $\big| D_p(\mathbb{T}_{\lambda_n}) \big|$, which is tight.
\end{proof}

We now move on to the proof of Proposition \ref{unimodularity}.

\begin{proof}
We claim that it is enough to prove the result for functions $f$ such that:
\begin{itemize}
\item[(i)]
there is an $r \geq 0$ such that, if the distance between $x$ and $y$ is greater than $r$, then $f(X,x,y)=0$,
\item[(ii)]
there is a $v >0$ such that, if one of the balls of radius $r$ centered at $x$ or $y$ has a volume greater than $v$, then $f(X,x,y)=0$,
\item[(iii)]
there is an $s>r$ such that $f(X,x,y)$ only depends on the intersection of the balls of radius $s$ around $x$ and $y$ in $X$, bipointed at $x$ and $y$,
\item[(iv)]
$f$ is bounded and uniformly continuous for the (bipointed) GHP distance.
\end{itemize}
Note that assumption (iv) makes sense even if we consider the GHP (and not local GHP) distance. Indeed, by assumption (iii), we may see $f(X,x,y)$ as a function of $B_s(X,x)$, so we do not need the local GHP distance. To show our claim, assume the theorem is true for any $f$ satisfying (i), (ii), (iii) and (iv). By the monotone convergence theorem, it is true for all indicator functions of open events satisfying (i), (ii) and (iii). By the monotone class theorem, it is true for the indicator functions of any event satisfying (i), (ii) and (iii). Now let $A$ be an event whose indicator function satisfies (i) and (ii). By the monotone class theorem again, the event $A$ can be approximated by events whose indicator function satisfies (i), (ii) and (iii). More precisely, for any $\eps>0$, there is an event $B$ whose indicator function satisfies (i), (ii) and (iii), and such that
\[ \mathbb{E} \Big[ \int_{\hp} \mathbbm{1}_{A \Delta B}(\hp, \rho, y) \,\mu_{\hp}(\mathrm{d}y) \Big] \leq \eps. \]
Hence, we can get rid of assumption (iii). We can then get rid of assumptions (i) and (ii) by monotone convergence. Therefore, the result is true for any indicator function, and finally for any nonnegative measurable function.

We now prove the theorem for a function $f$ satisfying (i), (ii), (iii) and (iv). The idea is to use the unimodularity of (a variant of) $\mathbb{T}_{\lambda_n}$, and take the scaling limit.
Let $(\lambda_n)$ be a sequence satisfying the assumptions of Theorem \ref{mainthm}. We first remark that the triangulations $\mathbb{T}_{\lambda_n}$ are invariant by rerooting along the simple random walk. This is the type-$I$ analog of (a part of) Proposition 9 in \cite{CurPSHIT}, and the proof is exactly the same. Moreover, invariance along the simple random walk and unimodularity are closely related by Proposition 2.5 of \cite{BCstationary}. More precisely, we write $\widehat{\mathbb{T}}_{\lambda_n}$ for the map $\mathbb{T}_{\lambda_n}$ biased by the inverse of the degree of its root $\rho_n$. Then $\widehat{\mathbb{T}}_{\lambda_n}$ is unimodular. Hence, we have
\[ \mathbb{E} \left[ \frac{1}{3n^4} \sum_{y \in \mathbb{T}_{\lambda_n}} f \Big( \frac{1}{n} \widehat{\mathbb{T}}_{\lambda_n}, \rho_n, y \Big) \right] = \mathbb{E} \left[ \frac{1}{3n^4} \sum_{y \in \mathbb{T}_{\lambda_n}} f \Big( \frac{1}{n} \widehat{\mathbb{T}}_{\lambda_n}, y, \rho_n \Big) \right], \]
where $\rho_n$ is the root vertex of $\widehat{\mathbb{T}}_{\lambda_n}$. We may restrict ourselves to $y \in B_{rn} (\widehat{\mathbb{T}}_{\lambda_n})$ thanks to assumption (i). By the definition of $\widehat{\mathbb{T}}_{\lambda_n}$, the last equation can be rewritten
\[ \mathbb{E} \left[ \frac{1}{ \deg (\rho_n)} \frac{1}{3n^4} \sum_{y \in B_{rn} (\mathbb{T}_{\lambda_n})} f \Big( \frac{1}{n} \mathbb{T}_{\lambda_n}, \rho_n, y \Big) \right]= \mathbb{E} \left[ \frac{1}{ \deg (\rho_n)} \frac{1}{3n^4} \sum_{y \in B_{rn} (\mathbb{T}_{\lambda_n})} f \Big( \frac{1}{n} \mathbb{T}_{\lambda_n},y, \rho_n \Big) \right]. \]

Hence, in order to prove the proposition, it is enough to prove that the left-hand side converges as $n \to +\infty$ to the left-hand side in the statement of the proposition, multiplied by $\mathbb{E} \big[ \frac{1}{D} \big]$. The proof of the same fact for the right-hand side will be similar. By the Skorokhod representation theorem, we may assume the convergence in Lemma \ref{rootandLGHP} is almost sure. By the dominated convergence theorem (the domination follows from assumption (ii) and the fact that $f$ is bounded), it is enough to prove
\begin{equation}
\frac{1}{3n^4} \sum_{y \in B_{rn} (\mathbb{T}_{\lambda_n})} f \Big( \frac{1}{n} \mathbb{T}_{\lambda_n}, \rho_n, y \Big) \xrightarrow[n \to +\infty]{a.s.}  \int_{B_r(\hp)} f(\hp, \rho, y) \mu_{\hp}(dy).
\end{equation}
This follows from assumptions (iii) and (iv) and the GHP convergence of $\frac{1}{n} B_{(r+s)n}(\mathbb{T}_{\lambda_n})$ to $B_{r+s}(\hp)$.
\end{proof}

\begin{rem}
The same result is true for the Brownian plane. The proof is essentially the same, the only change is that we need to use Proposition \ref{convergenceBP} instead of Theorem \ref{mainthm}.
\end{rem}

\subsection{Identification of the perimeter and volume processes}

We first explain what we mean by biasing a process by a martingale, which will be used a lot in what follows. Let $X$ be an adapted process, and let $M$ be a martingale for the underlying filtration with $\mathbb{E} \big[ M_0 \big]=1$. We say that a process $Y$ is the process $X$ biased by $M$ if for all $r_0>0$, the process $(Y_r)_{0 \leq r \leq r_0}$ has the same distribution as $(X_r)_{0 \leq r \leq r_0}$ biased by $M_{r_0}$. The martingale property of $M$ shows the consistency for different values of $r_0$.

The hyperbolic Brownian plane is a biased version of the Brownian plane. Hence, it is naturally equipped with a perimeter and a volume process inherited from those of the Brownian plane. We denote by $\big( P_r^h \big)_{r \geq 0}$ and $\big( V_r^h \big)_{r \geq 0}$ the perimeter and volume processes of $\hp$. More precisely, the proof of Theorem \ref{mainthm} gives the following joint convergences in distribution:
\[
\left \{
\begin{array}{lclcl}
\frac{1}{n} \mathbb{T}_{\lambda_n} &\xrightarrow[n \to +\infty]{}& \hp,\\
\Big( \frac{1}{n^2} \big| \partial B_{rn}^{\bullet}(\mathbb{T}_{\lambda_n}) \big| \Big)_{r \geq 0} &\xrightarrow[n \to +\infty]{}& \big( \frac{3}{2} P_r^h \big)_{r \geq 0}, \\
\Big( \frac{1}{n^4} \big| B_{rn}^{\bullet}(\mathbb{T}_{\lambda_n}) \big| \Big)_{r \geq 0} &\xrightarrow[n \to +\infty]{}& \big( 3 V_r^h \big)_{r \geq 0}.
\end{array}
\right.
\]
As previously, the first convergence is for the local GHP distance and the other two for the Skorokhod topology. Moreover, for all $r_0>0$, the triplet $\Big( \overline{B_{r_0}}(\hp), \big( P_r^h \big)_{0 \leq r \leq r_0}, \big( V_r^h \big)_{0 \leq r \leq r_0}  \Big)$ has the same distribution as $\Big( \overline{B_{r_0}}(\p), \big( P_r \big)_{0 \leq r \leq r_0}, \big( V_r \big)_{0 \leq r \leq r_0}  \Big)$ biased by $\varphi(P_{2r_0}, V_{2r_0})$.

This description implies that these two triplets have the same a.s. properties as for the Brownian plane. In particular, $P^h$ and $V^h$ are both càdlàg processes and can be expressed as measurable functions of $\hp$. Indeed, we have $V_r^h=|\overline{B_r}(\hp)|$ for all $r \geq 0$, and Proposition 1.1 of \cite{CLGHull} (already recalled in \eqref{perimeter_as_limit}) gives the convergence
\[ \frac{1}{\eps^2} |B_{r+\eps} (\hp) \backslash \overline{B_r} (\hp) | \xrightarrow[\eps \to 0]{} P_r^h\]
in probability.

On the other hand, the perimeter and volume of the map $B_{rn}^{\bullet}(\mathbb{T}_{\lambda_n})$ only depend on $B_{rn}^{\bullet}(\mathbb{T}_{\lambda_n})$. Hence, if we apply the proof of Theorem \ref{mainthm} to the pair $\Big( \frac{1}{n^2} \big| \partial B_{rn}^{\bullet}(\mathbb{T}_{\lambda_n}) \big|, \frac{1}{n^4} \big| B_{rn}^{\bullet}(\mathbb{T}_{\lambda_n}) \big| \Big)_{r \geq 0}$ instead of the triplet, we only need to bias by $\varphi(P_{r_0}, V_{r_0})$ instead of $\varphi(P_{2r_0}, V_{2r_0})$. We obtain the following result.

\begin{lem}\label{biasedtwoprocesses}
The pair of processes $(P^h, V^h)$ has the same distribution as $(P,V)$ biased by $\varphi (P,V)$.
\end{lem}

The goal of this section is to identify the two processes $P^h$ and $V^h$ in a more convenient way as expressed in Theorem \ref{identification}. Before moving on to the proof of Theorem \ref{identification}, we recall the description of $(P,V)$ given by \cite{CLGHull}. Let $Z$ be the critical continuous-state branching process with branching mechanism $\sqrt{\frac{8}{3}} \lambda^{3/2}$ for $\lambda>0$. We also recall that the measure $\nu$ is defined by $\nu(\mathrm{d}x)=\frac{e^{-1/2x}}{\sqrt{2 \pi x^5}} \mathbbm{1}_{x>0}\,\mathrm{d}x$. Then we can restate some results of \cite{CLGHull} as follows.

\begin{thm}[\cite{CLGHull}, Proposition 1.2 (ii) and Theorem 1.3]\label{CLG_identification}\
\vspace{-2mm}
\begin{itemize}
\item[1)]
The perimeter process $P$ of the Brownian plane has the same distribution as the time-reversal of $Z$, started from $+\infty$ at time $-\infty$, and conditioned to die at time $0$.
\vspace{-1mm}
\item[2)]
Conditionally on $P$, the process $V$ has the same distribution as the process
\[ \Big( \sum_{t_i \leq r} \left( \Delta P_{t_i} \right)^2 \xi_i \Big)_{r \geq 0} ,\]
where $(t_i)$ is a measurable enumeration of the jumps of $P$, and the variables $\xi_i$ are i.i.d. with distribution $\nu$.
\end{itemize}
\end{thm}

We will first prove the second part of Theorem \ref{identification}, that is, we determine the distribution of $V^h$ conditionally on $P^h$. We recall that for all $\delta>0$, the measure $\nu_{\delta}$ on $\R$ is defined by 
\[ \nu_{\delta}(\mathrm{d}x)= \frac{\delta^3 e^{2 \delta}}{1+2 \delta} \frac{e^{-\frac{\delta^2}{2x}-2x}}{\sqrt{2 \pi x^5}} \mathbbm{1}_{x>0} \, \mathrm{d}x.\]
If $\xi^h(\delta)$ is a random variable with distribution $\nu_{\delta}$, we have, for all $\beta \geq 0$,
\begin{equation}\label{Laplacenu}
\mathbb{E} \big[ e^{-\beta \xi^h(\delta)} \big]=\frac{(1+\delta \sqrt{4+2\beta}) e^{-\delta \sqrt{4+2\beta}}}{(1+2 \delta)e^{-2\delta}}.
\end{equation}
Notice for further use that $\xi^h(\delta)$ has the same distribution as $\delta^2 \xi$ biased by $e^{-2\delta^2 \xi}$, where $\xi$ is a random variable with density $\nu$.

\begin{proof}[Proof of Theorem \ref{identification}]
We fix $r_0>0$. We write $D([0,r_0])$ for the Skorokhod space on $[0,r_0]$. We write $g(p)=e^p \int_0^1 e^{-3x^2 p} \mathrm{d}x$. Let $t_1, t_2, \dots$ (resp. $t_1^h, t_2^h, \dots$) be the (random) jump times of the process $P$ (resp. $P^h$) up to time $r_0$, ordered by decreasing size of jumps $|\Delta P_{t_i}|$. Let $f$ be any measurable function $f: D([0,r_0]) \longrightarrow \R^+$ and let $u_1, u_2, \dots \geq 0$. Since $V$ is a pure jump process and only jumps at jump times of $P$, by Lemma \ref{biasedtwoprocesses} we have
\begin{equation} \label{vh_to_v}
\mathbb{E} \Big[ f \big( (P^h_r)_{0 \leq r \leq r_0} \big) \exp \Big( - \sum_{i \geq 0} u_i |\Delta V^h_{t_i^h}| \Big) \Big] = \mathbb{E} \Big[  f \big( (P_r)_{0 \leq r \leq r_0} \big) g(P_{r_0})  \exp \Big( - \sum_{i \geq 0} (u_i+2) |\Delta V_{t_i}| \Big)  \Big].
\end{equation}

By Theorem \ref{CLG_identification}, conditionally on $P$, the jumps $\Delta V_{t_i}$ are independent, distributed as $\big( \Delta P_{t_i} \big)^2 \xi$, where $\xi$ has law $\nu$. Hence, the last display can be rewritten as
\begin{equation} \label{vh_to_v2}
\mathbb{E} \bigg[  f \big( (P_r)_{0 \leq r \leq r_0} \big) g(P_{r_0}) \prod_{i \geq 0} \mathbb{E} \big[ e^{-2 (\Delta P_{t_i})^2 \xi} | \Delta P_{t_i} \big]
\frac{\mathbb{E} \big[ e^{-(u_i+2) (\Delta P_{t_i})^2 \xi} | \Delta P_{t_i}  \big]}
{\mathbb{E} \big[ e^{-2 (\Delta P_{t_i})^2 \xi} | \Delta P_{t_i}  \big]}
 \bigg].
\end{equation}

From the distribution of $\xi$ we compute easily, for $\alpha \geq 0$,
\[ \mathbb{E} \big[ e^{-2 \alpha^2 \xi} \big]=(1+2\alpha)e^{-2\alpha}.\]
By combining this with Equation \eqref{Laplacenu}, \eqref{vh_to_v2} becomes
\[ \mathbb{E} \Big[ f \big( (P_r)_{0 \leq r \leq r_0} \big) g(P_{r_0}) \prod_{i \geq 0} (1+2|\Delta P_{t_i}| ) e^{-2 |\Delta P_{t_i}|} \prod_{i \geq 0} \mathbb{E} \big[ e^{-u_i \xi^h(|\Delta P_{t_i}|)} | \Delta P_{t_i}  \big] \Big],\]
where $\xi^h(\delta)$ has law $\nu_{\delta}$.

This expression shows two things. First, it proves point 2) of Theorem \ref{identification}, that is, conditionally on $(P^h_r)_{0 \leq r \leq r_0}$, the jumps $\Delta V^h_{t^h_i}$ are independent and of law $\nu_{|\Delta P^h_{t_i^h}|}$. Second, by setting $u_i=0$ for every $i$, it proves that the density of the process $(P_r^h)_{0 \leq r \leq r_0}$ with respect to $(P_r)_{0 \leq r \leq r_0}$ is given by
\begin{equation}\label{densityP}
g(P_{r_0}) \prod_{i \geq 0} (1+2 |\Delta P_{t_i}|)e^{-2|\Delta P_{t_i}|}.
\end{equation}

We now move on to characterizing the law of $P^h$ in a nicer way. Since we know by Theorem \ref{CLG_identification} that $P$ is a reversed branching process with mechanism $\sqrt{\frac{8}{3}} u^{3/2}$, it seems natural to first study the effect of the density \eqref{densityP} on its associated Lévy process.
So let $S$ be the spectrally \emph{positive} $\frac{3}{2}$-stable Lévy process. We normalize it in such a way that, for all $t,u \geq 0$, we have
\[\mathbb{E} \big[e^{-u S_t} \big]=\exp \Big( \sqrt{\frac{8}{3}} t u^{3/2} \Big)=\exp \Big(  t \int_0^{+\infty} (e^{-ux}-1+ux) \mu(\mathrm{d}x) \Big),\]
where $\mu(\mathrm{d}x)=\sqrt{\frac{3}{2\pi}} x^{-5/2} \mathbbm{1}_{x>0}\,\mathrm{d}x$.

Let $(t_i)$ be a measurable enumeration of the jumps of $S$ and $M_t=e^{-S_t} \prod_{t_i \leq t} (1+2 |\Delta S_{t_i}|) e^{-2 |\Delta S_{t_i}|}$ for all $t \geq 0$.
\begin{lem}
The process $M$ is a martingale with respect to the natural filtration associated to $S$.
\end{lem}

\begin{proof}
In the whole proof we will write $f(x)=(1+2x)e^{-2x}$. To prove the lemma, we first note that $f(x) \leq 1$ for all $x \geq 0$, so the product defining $M$ is always well-defined, and $\mathbb{E} \big[ M_t \big] \leq \mathbb{E} \big[ e^{-S_t} \big] <+\infty$ for all $t$. Let $(\mathscr{F}_t)_{t \geq 0}$ be the natural filtration associated to $S$. Since $S$ is a Lévy process, it is easy to see that, if $s \leq t$, then $\mathbb{E} \big[ M_t | \mathscr{F}_s \big]=M_s \mathbb{E} \big[ M_{t-s} \big]$. Hence, it is enough to prove that $\mathbb{E} \big[ M_t \big]=1$ for all $t \geq 0$.

We claim that, for all $u \geq 0$ and $t \geq 0$, we have
\begin{equation}\label{LaplaceM}
\mathbb{E} \Big[ e^{-u S_t} \prod_{t_i \leq t} f(\Delta S_{t_i}) \Big]=e^{t \psi(u)},
\end{equation}
where $\psi(u)=\int_0^{+\infty} (e^{-ux} f(x) -1 + ux) \mu(\mathrm{d}x)=\sqrt{\frac{8}{3}} \frac{u^2+u-2}{\sqrt{u+2}}$. Once this is proved, we have in particular $\psi(1)=0$ and the lemma follows. To prove \eqref{LaplaceM}, for $\eps>0$, we write
\[S_t^{\eps}=\underset{\underset{\Delta S_{t_i} \geq \eps}{t_i \leq t}}{\sum} \Delta S_{t_i}-t \int_{\eps}^{+\infty}x\,\mu(\mathrm{d}x).\]
We know that $S_t$ is the a.s. limit of $S_t^{\eps}$ as $\eps \to 0$. Moreover, for every $\eps$, we have $$\mathbb{E} \big[ e^{-2 u S_t^{\eps}} \big]=\exp \Big(  t \int_{\eps}^{+\infty} (e^{-2ux}-1+2ux) \mu(\mathrm{d}x) \Big) \leq \E \big[ e^{-2uS_t} \big].$$ Hence, since $f \leq 1$, the variables $e^{-u S^{\eps}_t} \prod_{\underset{\Delta S_{t_i} \geq \eps}{t_i \leq t}} f( \Delta S_{t_i})$ are bounded in $L^2$ as $\eps \to 0$. Therefore, they are uniformly integrable, so the left-hand side of \eqref{LaplaceM} is equal to
\[ \lim_{\eps \to 0} \mathbb{E} \Big[ e^{-u S^{\eps}_t} \prod_{\underset{\Delta S_{t_i} \geq \eps}{t_i \leq t}} f( \Delta S_{t_i}) \Big] = \lim_{\eps \to 0} e^{tu \int_{\eps}^{+\infty}x \mu(\mathrm{d}x)} \mathbb{E} \Big[ \prod_{\underset{\Delta S_{t_i} \geq \eps}{t_i \leq t}} e^{-u \Delta S_{t_i}} f( \Delta S_{t_i}) \Big].\]
By the exponential formula for Poisson point processes, the expectation in the right-hand side is equal to 
\[\exp \Big( t \int_{\eps}^{+\infty} \big( e^{-u x} f(x)-1 \big) \mu(\mathrm{d} x) \Big),\]
which proves \eqref{LaplaceM} by letting $\eps \to 0$.
\end{proof}

Since $M$ is a martingale with $M_0=e^{-S_0}$, we may consider the process $S$ biased by $e^{S_0} M$. We denote it by $S^h$. From the form of $M$ it is easy to prove that $S^h$ is a Lévy process. Moreover, by \eqref{LaplaceM}, it holds for all $u \geq 0$ that
\begin{equation} \label{exposantSh}
\mathbb{E} \big[ e^{-uS^h_t} \big]=\mathbb{E} \big[ M_t e^{-u S_t} \big]=e^{t \psi(u+1)}=\exp \big( \sqrt{\frac{8}{3}} u \sqrt{u+3} \big).
\end{equation}
We can also compute the jump measure of $S^h$, which is given by
\begin{equation}\label{muh}
\mu^h(\mathrm{d}x)=(1+2x)e^{-2x} \mu(dx)=\sqrt{\frac{3}{2\pi}} \frac{1+2x}{x^{5/2}} e^{-2x} \mathbbm{1}_{x>0}\,\mathrm{d}x.
\end{equation} 

In order to study the continuous-state branching process (CSBP) associated to $S^h$ via the Lamperti transform, we consider the process $S^h$ started from $x>0$ and write $\tau=\inf \{t | S^h_t=0 \}$. Note that $\tau$ is a.s. finite since $S^h$ drifts to $-\infty$ and has only positive jumps. We claim that $(S^h_{t \wedge \tau})_{t \geq 0}$ under $\mathbb{P}_x$ has density $e^x M_{\tau}$ with respect to $(S_{t \wedge \tau})_{t \geq 0}$ under $\mathbb{P}_x$. To prove it, we write $\tau_n=\frac{\lceil 2^n \tau \rceil}{2^n}$. For all $t_1, \dots, t_k \geq 0$ and $f: \R^k \to \R$ bounded, we have
\begin{eqnarray*}
\mathbb{E} \big[ f(S^h_{t_1 \wedge \tau}, \dots, S^h_{t_k \wedge \tau}) \big] &=& \lim_{n \to +\infty} \mathbb{E} \big[ f(S^h_{t_1 \wedge \tau_n}, \dots, S^h_{t_k \wedge \tau_n}) \big]  \\
&=&  \lim_{n \to +\infty} \sum_{i=1}^{\infty} \mathbb{E} \big[ f(S^h_{t_1 \wedge \tau_n}, \dots, S^h_{t_k \wedge \tau_n}) \mathbbm{1}_{\tau_n=i/2^n} \big]\\
&=& \lim_{n \to +\infty} \sum_{i=1}^{\infty} \mathbb{E} \big[e^x M_{i/2^n} f(S_{t_1 \wedge \tau_n}, \dots, S_{t_k \wedge \tau_n}) \mathbbm{1}_{\tau_n=i/2^n} \big]\\
&=& \lim_{n \to +\infty} \mathbb{E} \big[e^x M_{\tau_n} f(S_{t_1 \wedge \tau_n}, \dots, S_{t_k \wedge \tau_n}) \big]\\
&=& \mathbb{E} \big[e^x M_{\tau} f(S_{t_1 \wedge \tau}, \dots, S_{t_k \wedge \tau}) \big],
\end{eqnarray*}
which proves the claim.

We now introduce $Z^h$, the CSBP with branching mechanism $\psi^h (u)=\sqrt{\frac{8}{3}} u \sqrt{u+3}$ that is associated with $S^h$ via the Lamperti transform. We also recall that $Z$ is the CSBP with branching mechanism $\sqrt{\frac{8}{3}} u^{3/2}$. Since the Lamperti transform is a measurable function of the Lévy process, the process $Z^h$ started from $x$ has density 
\begin{equation}\label{density_z_zh}
e^x \prod_{t_i} (1+2|\Delta Z_{t_i}|)e^{-2|\Delta Z_{t_i}|}
\end{equation}
with respect to $Z$ started from $x$.

We will now do the same construction as Curien and Le Gall in Section 2.1 of \cite{CLGHull} with this new branching mechanism. The semigroup of $Z^h$ is characterized as follows: for all $\lambda>0$ and $x,t \geq 0$ we have
\[ \mathbb{E}_x \big[ e^{-\lambda Z^h_t}\big]=e^{-x u_t^h(\lambda)},\]
where
\begin{equation}\label{equadiff}
\frac{\mathrm{d}u_t^h(\lambda)}{\mathrm{d}t}=-\sqrt{\frac{8}{3}} u_t^h(\lambda) \sqrt{u_t^h(\lambda)+3} \quad \mbox{ and } \quad u_0(\lambda)=\lambda.
\end{equation}
The solution of this equation is
\[u_t^h(\lambda)=\frac{3}{\sinh^2 \Big( \argsh \sqrt{\frac{3}{\lambda}}+\sqrt{2} t \Big)}.\]
This gives
\[ \mathbb{E}_x \big[ e^{-\lambda Z^h_t}\big] = \exp \bigg( -\frac{3x}{\sinh^2 \Big( \argsh \sqrt{\frac{3}{\lambda}}+\sqrt{2} t \Big)} \bigg) \]
and, by differentiating with respect to $\lambda$,
\begin{equation} \label{derivelaplace}
\mathbb{E}_x \big[ Z^h_t e^{-\lambda Z^h_t}\big] = \frac{3\sqrt{3} x}{\lambda\sqrt{\lambda+3}} \frac{\cosh \Big(\argsh \sqrt{\frac{3}{\lambda}}+\sqrt{2} t \Big)}{\sinh^3 \Big( \argsh \sqrt{\frac{3}{\lambda}}+\sqrt{2} t \Big)} \exp \Big( -\frac{3x}{\sinh^2 \Big( \argsh \sqrt{\frac{3}{\lambda}}+\sqrt{2} t \Big)} \Big).
\end{equation}

Let $\tau'=\inf \{ t \geq 0 | Z^h_t=0 \}$ be the extinction time of $Z^h$. By the above calculation, we have
\[ \mathbb{P}_x \big( \tau' \leq t \big)=\mathbb{P} \big( Z^h_t=0 \big)=\lim_{\lambda \to +\infty} \mathbb{E} \big[ e^{-\lambda Z^h_t} \big]=\exp \Big(-\frac{3x}{\sinh ^2 ( \sqrt{2} t) } \Big).\]
Hence, $\tau'$ has density
\begin{equation}\label{Phitx}
\Phi^h_t(x)=6\sqrt{2} x \frac{\cosh (\sqrt{2} t)}{\sinh^3 (\sqrt{2} t)} \exp \Big(-\frac{3x}{\sinh ^2 ( \sqrt{2} t) } \Big)
\end{equation}
with respect to the Lebesgue measure.

We now introduce the process $Z^h$ conditioned on extinction at a fixed time $\rho$. We denote by $q_t(x,\mathrm{d}y)$ the transition kernels of $Z^h$. The process $Z^h$ conditioned on extinction at time $\rho$ is the time-inhomogeneous Markov process indexed by $[0,\rho]$ whose transition kernel between times $s$ and $t$ is
\[\pi_{s,t}(x,\mathrm{d}y)=\frac{\Phi^h_{\rho-t}(y)}{\Phi^h_{\rho-s}(x)} q_{t-s}(x,\mathrm{d}y)\] 
for $x>0$ and $0 \leq s<t<\rho$, and
\[ \pi_{s, \rho}(x,\mathrm{d}y)=\delta_0(\mathrm{d}y).\]
As in \cite{CLGHull}, this interpretation can be justified by the fact that, for all $0<s_1< \dots < s_k$, the conditional distribution of $(Z^h_{s_1}, \dots , Z^h_{s_k})$ on $\mathbb{P}_x \big( \cdot | \rho \leq \tau' \leq \rho+\eps \big)$ converges to
\[\pi_{0,s_1}(x,\mathrm{d}y_1) \pi_{s_1,s_2}(y_1,\mathrm{d}y_2) \cdots  \pi_{s_{k-1},s_k}(y_{k-1},\mathrm{d}y_k) \]
as $\eps \to 0^+$.

We also recall that the extinction time of $Z$ has density $\Phi_t(x)=\frac{3x}{t^3} \exp \Big( -\frac{3x}{2t^2} \Big)$ (see Section 2.1 of \cite{CLGHull}). By combining this with the density \eqref{density_z_zh} for the nonconditioned processes, we get an absolute continuity relation for conditioned versions of $Z$ and $Z^h$. The process $Z^h$ started from $x$ and conditioned to die at time $\rho$ has density
\begin{equation} \label{densityX}
e^x \frac{\Phi_{\rho}(x)}{\Phi^h_{\rho}(x)} \prod_{t_i \geq 0} (1+2|\Delta Z_{t_i}|)e^{-2|\Delta Z_{t_i}|}
\end{equation}
with respect to $Z$ started from $x$ and conditioned to die at time $\rho$. To prove this properly, we just need to condition on $\rho \leq \tau' \leq \rho+\eps$ and let $\eps$ go to $0$.

We finally define $X^h$, which is a version of $Z^h$, starting from $+\infty$ at time $-\infty$, and conditioned to die exactly at time $0$. We can construct a process $(X^h_t)_{t \leq 0}$ with càdlàg paths and no negative jumps such that:
\begin{itemize}
\item[(i)]
$X^h_t>0$ for all $t<0$ and $X^h_0=0$ a.s.,
\item[(ii)]
$X^h_t \longrightarrow +\infty$ almost surely as $t \to -\infty$,
\item[(iii)]
for all $x$, if $T_x=\inf \{ t \leq 0 | X_t \leq x\}$, the process $(X^h_{(T_x+t) \wedge 0})_{t \geq 0}$ has the same distribution as the process $Z^h$ started from $x$.
\end{itemize}
Note that \cite{CLGHull} describes a process $X$ that is obtained from $Z$ in the same way $X^h$ is obtained from $Z^h$. Our $Z$ corresponds to the $X$ of \cite{CLGHull}, whereas our $X$ corresponds to $\widetilde{X}$. As in \cite{CLGHull}, we can get an explicit construction of $X^h$ by concatenating independent copies of $Z^h$ started from $n+1$ and killed when hitting $n$ for every $n \in \N$.

Proposition 4.4 of \cite{CLGHull} states that for any $\rho, x>0$, the process $(X_{t-\rho})_{0 \leq t \leq \rho}$ conditioned on $X_{-\rho}=x$ has the same distribution as $Z$ started from $x$ and conditioned to die exactly at time $\rho$. By the same proof, this also holds for $X^h$ and $Z^h$. Hence, $(X^h_{-t})_{0 \leq t \leq \rho}$ conditioned on $X^h_{-\rho}=x$ has the density \eqref{densityX} with respect to $(X_{-t})_{0 \leq t \leq \rho}$ conditioned on $X_{-\rho}=x$. Therefore, by the first point of Theorem \ref{CLG_identification}, the process $(X^h_{-t})_{0 \leq t \leq \rho}$ conditioned on $X^h_{-\rho}=x$ has the density \eqref{densityX} with respect to $(P_r)_{0 \leq r \leq \rho}$ conditioned on $P_{\rho}=x$. But by \eqref{densityP}, the process $(P^h_r)_{0 \leq r \leq \rho}$ conditioned on $P^h_{\rho}=x$ has a density of the form
\[ f(x, \rho) \prod_{t_i} (1+2|\Delta Z_{t_i}|)e^{-2|\Delta Z_{t_i}|}\]
with respect to $(P_r)_{0 \leq r \leq \rho}$ conditioned on $P_{\rho}=x$. Since the density must have expectation one, we must have $f(x, \rho)=e^x \frac{\Phi_{\rho}(x)}{\Phi^h_{\rho}(x)}$.

Hence, in order to prove that $P^h$ has the same distribution as $(X^h_{-t})_{t \geq 0}$, we only need to prove that these two processes have the same one-dimensional marginals. To this end, we will now compute the Laplace transform of the one-dimensional marginals of $X^h$.

\begin{lem}\label{unedimY}
For all $t \geq 0$ and $\lambda \geq 0$, we have
\[\mathbb{E} \big[ e^{-\lambda X^h_{-t}}\big] = \Big( 1+ \frac{\lambda}{3} \sinh^2 (\sqrt{2} t) \Big)^{-1} \Big( 1+ \frac{\lambda}{3} \tanh^2 (\sqrt{2} t) \Big)^{-1/2}. \]
\end{lem}

\begin{proof}
Once again, the same computation for $X$ is performed in \cite{CLGHull}. By the exact same proof as in the beginning of the proof of Proposition 1.2 (ii) (section 4.1) in \cite{CLGHull}, we have
\[\mathbb{E} \big[ e^{-\lambda X^h_{-t}}\big] = \lim_{x \to +\infty} \mathbb{E}_x \Big[ \int_0^{+\infty} e^{-\lambda Z^h_s} \Phi^h_t (Z^h_s) \mathrm{d}s \Big]
= \lim_{x \to +\infty}  \int_0^{+\infty} \mathbb{E}_x \big[ e^{-\lambda Z^h_s} \Phi^h_t (Z^h_s) \big] \mathrm{d}s.\]

By \eqref{Phitx} and \eqref{derivelaplace}, one can compute $ \mathbb{E}_x \big[ e^{-\lambda Z^h_s} \Phi^h_t (Z^h_s) \big]$ exactly:
\begin{eqnarray*}
 \mathbb{E}_x \big[ e^{-\lambda Z^h_s} \Phi^h_t (Z^h_s) \big] &=& 6\sqrt{2} \frac{\cosh (\sqrt{2} t)}{\sinh^3 (\sqrt{2} t)} \mathbb{E}_x \Big[ Z^h_s \exp \Big( -\Big( \lambda+\frac{3}{\sinh^2 (\sqrt{2} t)} \Big) Z^h_s \Big) \Big]\\
&=& \frac{18\sqrt{6} x}{\Big( \lambda+\frac{3}{\sinh^2 (\sqrt{2} t)} \Big) \sqrt{\lambda+3 \coth^2 (\sqrt{2} t)}} \frac{\cosh (\sqrt{2} t)}{\sinh^3 (\sqrt{2} t)} \\
& \times & \frac{\cosh \Big(\argsh \sqrt{\frac{3 \sinh^2 (\sqrt{2} t)}{\lambda \sinh^2 (\sqrt{2} t) + 3}}+\sqrt{2} s \Big)}{\sinh^3 \Big(\argsh \sqrt{\frac{3 \sinh^2 (\sqrt{2} t)}{\lambda \sinh^2 (\sqrt{2} t) + 3}}+\sqrt{2} s \Big)}\\
& \times & \exp \Bigg( -\frac{3x}{\sinh^2 \Big( \argsh \sqrt{\frac{3 \sinh^2 (\sqrt{2} t)}{\lambda \sinh^2 (\sqrt{2} t) + 3}}+\sqrt{2} s \Big)} \Bigg).
\end{eqnarray*}

We can now integrate over $s \geq 0$ to get
\begin{eqnarray*}
\int_0^{+\infty} \mathbb{E}_x \big[ e^{-\lambda Z^h_s} \Phi^h_t (Z^h_s) \big] \mathrm{d}s &=& \frac{3 \sqrt{3}}{\Big( \lambda+\frac{3}{\sinh^2 (\sqrt{2} t)} \Big) \sqrt{\lambda+3 \coth^2 (\sqrt{2} t)}} \frac{\cosh (\sqrt{2} t)}{\sinh^3 (\sqrt{2} t)}\\
& \times & \Bigg( 1-\exp \Big( -x \Big( \lambda+\frac{3}{\sinh^2 (\sqrt{2} t)} \Big) \Big) \Bigg).
\end{eqnarray*}
As $x \to +\infty$, the last factor goes to $1$ and we get the claimed result.
\end{proof}

In remains to check that $P^h_t$ has indeed the same Laplace transform. This is obtained by combining Lemma \ref{biasedtwoprocesses} and the Laplace transforms of the variables $P_r$ and $V_r$ that are computed in Proposition 1.2 and 1.4 of \cite{CLGHull}. The computation is essentially the same as the proof of Lemma \ref{mass} and we omit it here. This completes the proof of Theorem \ref{identification}.
\end{proof}

\subsection{Asymptotics for the perimeter and volume processes}

\begin{proof}[Proof of Corollary \ref{asymptoticspv}]
Section 4.4 of \cite{CLGpeeling} gives another construction of the process $(P, V)$ via the Lamperti transform. We mimic this construction for $(P^h,V^h)$. We consider the Lévy process $S^h$ described by \eqref{exposantSh}, started from $x>0$. We write $\gamma$ for the first time at which $S^h$ hits $0$. For every $t \geq 0$, we also set
\[\tau_t=\inf \bigg\{ s \geq 0 \bigg| \int_0^s \frac{\mathrm{d}u}{S^h_u} \geq  t \bigg\}. \]
Finally, let $T=\int_0^{\gamma} \frac{\mathrm{d}u}{S^h_u}$. The process $\left( S^h_{\tau_t} \right)_{0 \leq t \leq T}$ is the \emph{Lamperti transform} of $S^h$, and has the same distribution as the branching process $Z^h$ started from $x$. We now introduce the time-reversal $\reflectbox{$S$}$ of $S^h$ (the notation $\reflectbox{$S$}$ is the mirror of $S$). It is the Lévy process with no positive jumps whose distribution is characterized by
\[\mathbb{E} \big[ e^{u \reflectbox{$S$}_t} \big]=\exp \Big( t \sqrt{\frac{8}{3}} u \sqrt{u+3} \Big) \mbox{\: for \:} u \geq 0.\]
Let also $\reflectbox{$S$}^+$ be the process $\reflectbox{$S$}$ conditionned to stay positive. We recall a classical time-reversal theorem (e.g. Theorem VII.18 of \cite{Ber96}). The process $S^h$ started from $x>0$ and killed when hitting $0$ is the time-reversal of the process $\reflectbox{$S$}^+$ stopped when hitting $x$ for the last time. By applying the Lamperti tranform, the time reversal of $X^h$, between its first passage at $x$ and $0$, is the Lamperti transform of $\reflectbox{$S$}^+$, stopped when hitting $x$ for the last time. But by item 1) of Theorem \ref{identification}, the time-reversal of $X^h$ is $P^h$. Hence, the process $P^h$ is the Lamperti transform of $\reflectbox{$S$}^+$. More precisely, for every $t \geq 0$, let
\[ \eta_t=\inf \bigg\{ s \geq 0 \bigg| \int_0^s \frac{\mathrm{d}u}{\reflectbox{$S$}^+_u} \geq  t \bigg\}. \]
Then $P^h$ has the same distribution as $\left( \reflectbox{$S$}^+_{\eta_t} \right)_{t \geq 0}$. Moreover, let $(s_i)$ be a measurable enumeration of the jumps of $\reflectbox{$S$}$. Conditionally on $\reflectbox{$S$}$, let $\big( \xi^h_i \big)$ be independent variables such that, for all $i$, the variable $\xi^h_i$ has distribution $\nu_{|\Delta \reflectbox{$S$}_{s_i}|}$. We write
\[L_t= \sum_{s_i \leq t} \xi^h_i. \]
Let also $L^+$ be the process obtained by performing the exact same operation on $\reflectbox{$S$}^+$ instead of $\reflectbox{$S$}$. Since the construction of $L^+$ from $\reflectbox{$S$}^+$ is the same as the construction of $V^h$ from $P^h$, we get
\[ \big( P_r^h, V_r^h \big)_{r \geq 0} \overset{(d)}{=} \big( \reflectbox{$S$}^+_{\eta_t}, L^+_{\eta_t} \big)_{t \geq 0}. \]

It is now easy to obtain an asymptotic estimation of $(P_r^h)$ through the study of $\reflectbox{$S$}$ and its conditioned version $\reflectbox{$S$}^+$. For all $t \geq 0$, we have
\begin{eqnarray*}
\mathbb{E} \big[ \reflectbox{$S$}_t \big] &=& \frac{d}{d u} \Big|_{u=0} \mathbb{E} \big[ e^{u \reflectbox{$S$}^+_t} \big]\\
&=& \frac{d}{d u} \Big|_{u=0} \exp \Big(t \sqrt{\frac{8}{3}} u \sqrt{u+3} \Big)\\
&=& 2 \sqrt{2} t.
\end{eqnarray*}

Moreover, it holds that $\mathbb{E} \big[ e^{u \reflectbox{$S$}_t} \big] <+\infty$ for all $t \geq 0$ and $u \geq -3$ (this follows easily from the definition of $S^h$ as $S$ biased by an explicit martingale and the fact that $\reflectbox{$S$}$ is its time-reversal). Hence, by a classic moderate deviation argument, we have almost surely $\frac{\reflectbox{$S$}_n}{n} = 2 \sqrt{2}+O(n^{-1/4})$ as $n \to +\infty$ for $n \in \N$. Moreover, let $c>0$ be such that $\mathbb{P} \big( |\reflectbox{$S$}_t| \leq c \big) \geq \frac{1}{2}$ for all $0 \leq t \leq 1$. By the strong Markov property, we have $\mathbb{P} \big( |\reflectbox{$S$}_1| \geq x-c \big) \geq \frac{1}{2} \mathbb{P} \big( \sup_{t \in [0,1]} |\reflectbox{$S$}_t| \geq x \big)$ for all $x>0$. Hence, the random variable $\sup_{t \in [0,1]} |\reflectbox{$S$}_t|$ has exponential tails on both sides, so for $n \in \N$ large enough and $0 \leq t \leq 1$, we have $|\reflectbox{$S$}_{n+t}-\reflectbox{$S$}_n| \leq n^{1/4}$. Hence, we have $\frac{\reflectbox{$S$}_t}{t} \overset{a.s.}{\underset{t \to +\infty}{=}} 2 \sqrt{2}+O(t^{-1/4})$. 

But the distribution of $(\reflectbox{$S$}^+_{t+1}-\reflectbox{$S$}^+_1)_{t \geq 0}$ is just that of $\reflectbox{$S$}$, conditioned on an event of positive probability ($\reflectbox{$S$}$ drifts to $+\infty$ so it has a positive probability never to hit $0$ after time $1$). Hence, we deduce from above that
\[\frac{\reflectbox{$S$}_t^+}{t} \overset{a.s.}{\underset{t \to +\infty}{=}} 2 \sqrt{2}+O(t^{-1/4}) \mbox{ }\mbox{, and so }\mbox{ } \frac{1}{\reflectbox{$S$}_t^+}  \overset{a.s.}{\underset{t \to +\infty}{=}} \frac{1}{2 \sqrt{2}\,t}+O(t^{-5/4}).\]
The integral of the error term converges, so there is a random variable $X_{\infty}$ such that $\int_0^s \frac{\mathrm{d}u}{\reflectbox{$S$}_u^+} = \frac{\ln{s}}{2\sqrt{2}} +X_{\infty}+o(1)$ a.s. It follows that $\frac{\ln{\eta_t}}{2\sqrt{2}} = t-X_{\infty}+o(1)$, so $\eta_t \sim e^{-2 \sqrt{2} X_{\infty}} e^{2\sqrt{2} t}$ a.s. We finally get
\[ \reflectbox{$S$}^+_{\eta_t} \sim 2 \sqrt{2} \eta_t \sim  2 \sqrt{2} e^{-2 \sqrt{2} X_{\infty}} e^{2 \sqrt{2} t} \]
a.s. when $t \to +\infty$, so there is a random variable $0<\mathcal{E}<+\infty$ such that $e^{-2\sqrt{2} r} P^h_r \xrightarrow[r \to +\infty]{a.s.} \mathcal{E}$.

To prove that $\frac{V_r^h}{P_r^h}$ converges a.s. to a deterministic constant, we first notice that $L$ is a nondecreasing Lévy process. By construction, we have
\[ \mathbb{E} \big[ L_1 \big]=\int_{0}^{+\infty} \mathbb{E} \big[ \xi^h(x) \big] \mu^h(\mathrm{d}x), \]
where $\mu^h$ is the Lévy measure of $S^h$ that is computed in \eqref{muh}, and $\xi^h(x)$ has distribution $\nu_x$. By derivating \eqref{Laplacenu} at $\beta=0$, one can compute $\mathbb{E} \big[ \xi^h(x) \big]=\frac{x^2}{2x+1}$. Since $\mu^h$ integrates $x^2$ near $0$ and has exponential tail, the last display is finite. By the law of large numbers, we get
\[ \frac{L_t}{t} \xrightarrow[t \to +\infty]{a.s.} c, \]
where $c=\mathbb{E} \big[ L_1 \big]$. By absolute continuity on $[1, +\infty[$, the same holds for $L^+$ instead of $L$.

In order to complete the proof of Corollary \ref{asymptoticspv}, we only need to find the distribution of $\mathcal{E}$ and the constant $c$. We will do this with Laplace transforms. We know that, for all $\lambda, \mu \geq 0$, we have
\[ \mathbb{E} \big[ e^{-\lambda \mathcal{E}-c\mu \mathcal{E} } \big] = \lim_{r \to +\infty} \mathbb{E} \Big[ \exp \big( -\lambda e^{-2\sqrt{2} r} P_r^h - e^{-2 \sqrt{2} r} V_r^h \big) \Big].\]
But we can compute these Laplace transforms thanks to Lemma \ref{biasedtwoprocesses} and Proposition 1.2 and 1.4 of \cite{CLGHull}:
\begin{eqnarray*}
\mathbb{E} \big[ e^{- \lambda P_r^h - \mu V_r^h} \big] &=& \mathbb{E} \big[ e^{-\lambda P_r} e^{- \mu V_r} e^{P_r} e^{-2V_r} \int_0^1 e^{-3x^2P_r} \mathrm{d}x\big]\\
&=& \int_0^1 \mathbb{E} \Big[ e^{(1-\lambda-3x^2)P_r} \mathbb{E} \big[ e^{-(\mu+2)V_r} | P_r \big]\Big]\mathrm{d}x\\
&=& \Big(1+\frac{2 \lambda -2 +\sqrt{2 \mu+4}}{3 \sqrt{2 \mu+4}} \sinh^{2}{((2\mu+4)^{1/4}r)} \Big)^{-1}\\
& & \times \Big(1+\frac{2 \lambda +4 -2\sqrt{2 \mu+4}}{3 \sqrt{2 \mu+4}} \tanh^{2}{((2\mu+4)^{1/4}r)} \Big)^{-1/2}.
\end{eqnarray*}
This gives the value of $\mathbb{E} \Big[ \exp{ \Big(- e^{-2\sqrt{2}r} \lambda P_r^h -e^{-2\sqrt{2}r}\mu V_r^h\Big)} \Big]$. When we let $r$ go to $+\infty$, the second factor goes to $1$. We also have
\[\frac{2 \lambda e^{-2\sqrt{2}r} -2 +\sqrt{2 \mu e^{-2\sqrt{2}r} +4}}{3\sqrt{2 \mu e^{-2\sqrt{2}r} +4}} \sim \frac{\big(2 \lambda +\mu/2 \big) e^{-2\sqrt{2}r}}{6}\]
and $\sinh^{2}{\big( (2\mu e^{-2\sqrt{2}r}+4)^{1/4}r \big)} \sim \Big( \frac{e^{\sqrt{2}r}}{2} \Big)^2$, so the first factor goes to $\Big( 1+\frac{\lambda}{12}+\frac{\mu}{48} \Big)^{-1}$. This is the Laplace transform of the couple $\Big( X, \frac{1}{4}X\Big)$, where $X$ is an exponential variable of parameter $12$. This ends the proof.
\end{proof}

\appendix

\section{Proof of Proposition \ref{GHP}}

The goal of this appendix is to prove Proposition \ref{GHP}. We note that similar ideas to those below appear in Section 2 of \cite{MS15}, and more precisely in the proofs of Proposition 2.17 and 2.18. In particular, Lemma \ref{inclusion} is essentially proved in the proof of Proposition 2.18 there. If $A$ is a subset of a metric space $X$ and $\eps>0$, we will write $A^{\eps}$ for the union of all the open balls of radius $\eps$ centered at an element of $A$. Note that this is an open subset of $X$. We also recall that $B_r(X)$ and $\overline{B_r}(X)$ are respectively the closed ball and the hull centered at the root of $X$. In particular, they are both closed subsets of $X$.
To prove Proposition \ref{GHP}, we will also need several times the following definition.

\begin{defn}
Let $(X,d)$ be a metric space, $x,y \in X$ and $\eps>0$. An \textit{$\eps$-chain from $x$ to $y$} is a finite sequence of points $(z_i)_{0 \leq i \leq k}$ of $X$ such that $z_0=x$, $z_k=y$ and $d(z_i,z_{i+1}) \leq \eps$ for all $0 \leq i \leq k-1$.
\end{defn}

\begin{lem}\label{Hausdorff}
Under the assumptions of Proposition \ref{GHP}, the application $s \to \overline{B_s}(X)$ is continuous at $r$ for the Hausdorff distance on the set of the compact subsets of $\overline{B_{r+1}}(X)$.
\end{lem}

\begin{proof}
Let $\delta>0$. It is enough to show that, for some $\eps>0$, we have:
\begin{itemize}
\item[a)]
$\overline{B_{r+\eps}}(X) \subset \overline{B_r}(X)^{\delta}$,
\item[b)]
$\overline{B_r}(X) \subset \overline{B_{r-\eps}}(X)^{\delta}$.
\end{itemize}

We start with the first point. Let $A=\Big( \bigcap_{\eps>0} \overline{B_{r+\eps}}(X) \Big) \backslash \overline{B_r}(X)$. If $x \in A$, then there is a geodesic from $x$ to a point $y \in X \backslash \overline{B_{r+1}}(X)$ that stays outside $B_r(X)$. However, $\gamma$ has to intersect $B_{r+\eps}(X)$ for all $\eps>0$. This is clearly impossible, so $A=\emptyset$. In particular, the decreasing intersection of the compact sets $\overline{B_{r+1/n}}(X) \backslash \overline{B_r}(X)^{\delta}$ is empty, so one of these compact sets is empty, which proves item a).

For item b), let $A'=\big( \overline{B_r}(X)\backslash B_r(X) \big) \backslash \Big( \bigcup_{\eps>0} \overline{B_{r-\eps}}(X) \Big)$. By assumption $(iii)$ we have $\mu(A')=\emptyset$ and $A'$ is open, so by assumption $(ii)$ we get $A'=\emptyset$. This implies $ \overline{B_r}(X)\backslash B_r(X) \subset \bigcup_{n \geq 0} \overline{B_{r-1/n}}(X)^{\delta}$. Moreover, we have $B_r(X) \subset \overline{B_{r-1/n}}(X)^{\delta}$ for $\frac{1}{n} < \delta$. Hence, the increasing family of open sets $\overline{B_{r-1/n}}(X)^{\delta}$ covers the compact space $\overline{B_r}(X)$, so there is an $\eps>0$ such that $\overline{B_r}(X) \subset \overline{B_{r-\eps}}(X)^{\delta}$.
\end{proof}

\begin{rem}
Note that the first point is true even without assumptions $(ii)$ and $(iii)$.
\end{rem}

\begin{lem} \label{inclusion}
Let $\eps>0$. Let $(X, \rho, a)$ and $(Y, \rho, b)$ be two bipointed connected, compact subsets of a locally compact metric space $(Z,d)$. Assume that the Hausdorff distance between $X$ and $Y$ is less than $\eps$ and that $d(a,b) \leq \eps$. Then, for all $r$ such that $r+4 \eps < d(\rho,b)$, we have
\begin{itemize}
\item[1)]
$\overline{B_r}(X)^{\eps} \cap Y \subset \overline{B_{r+4\eps}}(Y)$,
\item[2)]
$\overline{B_r}(X) \subset \overline{B_{r+4\eps}}(Y)^{\eps}$.
\end{itemize}
\end{lem}

\begin{figure}
\begin{center}
\begin{tikzpicture}
\draw[red, thick] (0,0) to[out=135, in=270] (-2,2.5);
\draw[red, thick] (-2,2.5) to[out=90, in=180] (-0.5,4);
\draw[red, thick] (-0.5,4) to[out=0, in=180] (0.5,2);
\draw[red, thick] (0.5,2) to[out=0, in=180] (1.5,5);
\draw[red, thick] (1.5,5) to[out=0, in=30] (0,0);

\draw[blue, thick] (0,0) to[out=150, in=270] (-1.7,2.5);
\draw[blue, thick] (-1.7,2.5) to[out=90, in=180] (-0.5,3.5);
\draw[blue, thick] (-0.5,3.5) to[out=0, in=180] (0,2.3);
\draw[blue, thick] (0,2.3) to[out=0, in=180] (1.5,5.5);
\draw[blue, thick] (1.5,5.5) to[out=0, in=20] (0,0);

\draw[red] (2, 1.5) node[texte]{$Y$};
\draw[blue] (2.8, 3.5) node[texte]{$X$};
\draw (0, 0) node{};
\draw (0.2, -0.2) node[texte]{$\rho$};

\draw (1.5, 5.2) node[bleu]{};
\draw [blue] (1.8,5.2) node[texte]{$a$};
\draw (1.7, 4.8) node[rouge]{};
\draw [red] (2,4.6) node[texte]{$b$};

\draw (-0.6, 3.7) node[rouge]{};
\draw [red] (-0.9,3.7) node[texte]{$y$};
\draw (-0.8, 3.3) node[bleu]{};
\draw [blue] (-1,3.1) node[texte]{$x$};

\draw (-0.1,2.6) node[petitrouge]{};
\draw [red] (-0.5,2.6) node[texte]{$z_1$};
\draw (1,2) node[petitrouge]{};
\draw [red] (1,1.7) node[texte]{$z_2$};
\draw (1.6,3.5) node[petitrouge]{};
\draw [red] (1.9,3.4) node[texte]{$z_3$};

\draw (-0.3,2.2) node[petitbleu]{};
\draw [blue] (-0.5,1.9) node[texte]{$w_1$};
\draw (0.8,2.5) node[petitbleu]{};
\draw [blue] (1.3,2.5) node[texte]{$w_2$};
\draw (1.5,3.8) node[petitbleu]{};
\draw [blue] (1.8,3.9) node[texte]{$w_3$};

\draw[dashed, thick] (-3,2.9)--(3.5,2.9);
\draw (3.3,2.6) node[texte]{$r+4\eps$};
\end{tikzpicture}
\end{center}
\caption{Illustration of the proof of Lemma \ref{inclusion}} \label{chains}
\end{figure}
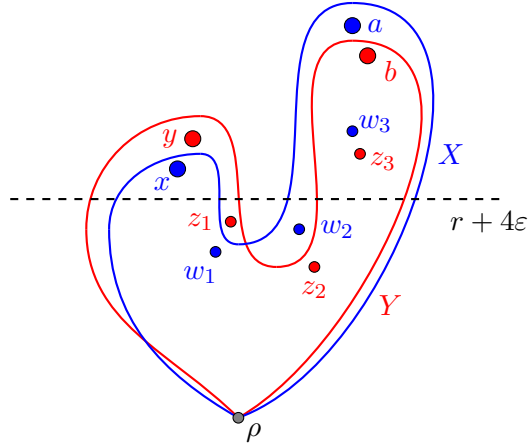

\begin{proof}
We first notice that the connectedness of $X$ implies that, for any two points $x$ and $x'$ in $X$, there is an $\eps$-chain in $X$ from $x$ to $x'$. The same is true for $Y$.

Let $y$ be a point in $\overline{B_r}(X)^{\eps} \cap Y$. We want to show that $y \in \overline{B_{r+4\eps}}(Y)$. We can assume $d(\rho, y)>r+4\eps$ (if it is not the case, then $y \in \overline{B_{r+4\eps}}(Y)$ is obvious). Let $(z_i)_{0 \leq i \leq k}$ be an $\eps$-chain from $y$ to $b$ in $Y$. We know that $d(a,b) \leq \eps$ and that there is a point $x \in X$ such that $d(x,y)\leq \eps$. We write $w_0=x$, $w_k=a$ and, for all $1 \leq i \leq k-1$, we take $w_i \in X$ such that $d(w_i,z_i) \leq \eps$ (see Figure \ref{chains} for an illustration). For all $i$, we have $d(w_i,w_{i+1}) \leq d(w_i,z_i)+d(z_i,z_{i+1})+d(z_{i+1},w_{i+1}) \leq 3 \eps$, so $(w_i)$ is a $3\eps$-chain from $x$ to $a$ in $X$. Since $x \in \overline{B_r}(X)$, there must be at least one $i$ such that $d(\rho,w_i) \leq r+3\eps$, which implies $d(\rho,z_i) \leq r+4\eps$. Hence, every $\eps$-chain from $y$ to $b$ has at least one point at distance from the root less than $r+4\eps$, so there is no $\eps$-chain from $y$ to $b$ in $Y \backslash B_{r+4 \eps}(Y)$. This implies that $y$ and $b$ do not lie in the same connected component of $Y \backslash B_{r+4 \eps}(Y)$, so $y \in \overline{B_{r+4\eps}}(Y)$, which proves the first point of the lemma.

The second point is easily obtained from the first one: if $x \in \overline{B_r}(X)$, then there is a $y \in Y$ such that $d(x,y) \leq \eps$. By the first point, we have $y \in \overline{B_{r+4\eps}}(Y)$, so $x \in  \overline{B_{r+4\eps}}(Y)^{\eps}$.
\end{proof}

\begin{proof}[Proof of Proposition \ref{GHP}]
First, we need to prove that the radii of the $\overline{B_r}(X_n)$ are bounded. Let $R=\max \{ d(\rho,x) | x \in \overline{B_{r+4}}(X)\}$ be the radius of $\overline{B_{r+4}}(X)$. For $n$ large enough, we have
\[d_{GH}(B_{R+3}(X_n),B_{R+3}(X))<1.\]
We write $X'=B_{R+3}(X)$ and $X'_n=B_{R+3}(X_n)$. The above inequality means that we can embed $X'$ and $X'_n$ isometrically in the same space $(Z,d)$ in such a way that $X'_n \subset (X')^1$ and $\rho_n=\rho$.

Let $b \in X'_n \backslash B_{R+2}(X_n)$ and $a \in X'$ such that $d(a,b) \leq 1$. We have $d(a,\rho) \geq R+1$, so $a \notin \overline{B_r}(X)$ and $\overline{B_r}(X)=\overline{B_r}(X')$ is the hull of radius $r$ with center $\rho$ with respect to $a$ in $X'$. By item 2) of Lemma \ref{inclusion} for $\eps=1$, we have $\overline{B_r}(X'_n) \subset \overline{B_{r+4}}(X')^1$. In particular, the radius of $\overline{B_r}(X'_n)$ is less than $R+1$. This means that only one connected component of $X' \backslash B_r(X'_n)$ contains points at distance greater than $R+1$ from the root. In other words, the radius of $\overline{B_r}(X_n)$ is at most $R+1$.

We now move on to the proof of our proposition. Let $\delta>0$. By Lemma \ref{Hausdorff}, there is $\eps>0$ such that $\overline{B_{r+4 \eps}}(X) \subset \overline{B_r}(X)^{\delta}$ and $\overline{B_r}(X) \subset \overline{B_{r-4 \eps}}(X)^{\delta}$. For $n$ large enough, we have $d_{GHP}(X',X'_n) \leq \eps$. This means that we can embed $X'$ and $X'_n$ in the same space $Z$ in such a way that
\begin{itemize}
\item[a)]
$\rho=\rho_n$,
\item[b)]
$X'_n \subset (X')^{\eps}$,
\item[c)]
$X' \subset (X'_n)^{\eps}$,
\item[d)]
for every $A \subset X'_n$ that is measurable we have $\mu_n(A) \leq \mu(A^{\eps})+\eps$,
\item[e)]
for every $B \subset X'$ that is measurable we have $\mu(B) \leq \mu_n(B^{\eps})+\eps$.
\end{itemize}
This embedding provides a natural way to embed the measured metric spaces $\overline{B_r}(X_n)$ and $\overline{B_r}(X)$ in $Z$. We will deduce an upper bound for the GHP distance between these two hulls.

For all $y \in \overline{B_r}(X_n)$ there is an $x \in X'$ such that $d(x,y) \leq \eps$. By Lemma \ref{inclusion} we have $x \in \overline{B_{r+4\eps}}(X)$, so by Lemma \ref{Hausdorff} and our choice of $\eps$, we have $x \in \overline{B_r}(X)^{\delta}$. This proves $\overline{B_r}(X_n) \subset \overline{B_r}(X)^{\delta+\eps}$.

Similarly, let $x \in \overline{B_r}(X)$. We have $x \in \overline{B_{r-4\eps}}(X)^{\delta}$ by Lemma \ref{Hausdorff} and our choice of $\eps$. Let $z \in \overline{B_{r-4\eps}}(X)$ be such that $d(x,z) \leq \delta$. There is a $y \in X'_n$ such that $d(y,z) \leq \eps$. By Lemma \ref{inclusion} we have $y \in \overline{B_{r-4\eps+4\eps}}(X_n)$ and $d(x,y) \leq \delta+\eps$, which proves $\overline{B_r}(X) \subset \overline{B_r}(X_n)^{\delta+\eps}$. Hence, in our embedding, the Hausdorff distance between $\overline{B_r}(X)$ and $\overline{B_r}(X_n)$ is less than $\eps+\delta$.

For all $A \subset \overline{B_r}(X_n)$ measurable, we have $\mu_n(A) \leq \mu (A^{\eps}) +\eps=\mu \big( A^{\eps} \cap X' \big) +\eps$. By Lemma \ref{inclusion} we have the inclusion $A^{\eps} \cap X' \subset A^{\eps} \cap \overline{B_{r+4\eps}}(X)$, so we get
\begin{eqnarray*}
\mu_n(A) & \leq & \eps+\mu \big( A^{\eps} \cap \overline{B_{r+4\eps}}(X) \big)\\
& \leq & \eps+ \mu \big( A^{\eps} \cap \overline{B_r}(X) \big) + V(r+4\eps)-V(r),
\end{eqnarray*}
where we recall that $V(s)=\mu(\overline{B_s}(X))$ for all $s$.

Similarly, for all $B \subset \overline{B_r}(X)$ measurable, we have
\begin{eqnarray*}
\mu(B) & \leq & \mu \big( B \cap \overline{B_{r-4\eps}}(X) \big) + V(r) - V(r-4\eps)\\
& \leq & \eps + \mu_n \big( (B \cap \overline{B_{r-4\eps}}(X))^{\eps} \cap X'_n \big)+V(r) - V(r-4\eps)\\
& \leq & \eps + \mu_n \big( B^{\eps} \cap \overline{B_{r-4\eps}}(X)^{\eps} \cap X'_n \big)+V(r) - V(r-4\eps)\\
& \leq & \eps + \mu_n \big( B^{\eps} \cap \overline{B_r}(X_n) \big) + V(r) - V(r-4\eps),
\end{eqnarray*}
where the last inequality uses Lemma \ref{inclusion}.

Hence, our embedding of $\overline{B_r}(X)$ and $\overline{B_r}(X_n)$ gives the following bound for $n$ large enough:
\[d_{GHP} \big( \overline{B_r}(X), \overline{B_r}(X_n) \big) \leq \max \Big( \eps+\delta, \eps+V(r+4\eps)-V(r), \eps+ V(r) - V(r-4\eps) \Big).\]

By assumption $(iii)$ in the statement of the proposition, the right-hand side can be made arbitrarily small, which proves the first point of Proposition \ref{GHP}. The second point is an obvious consequence of the first one.

Finally, in the case of compact, bipointed metric spaces  $\big( (X_n, d_n), x_n,y_n, \mu_n \big)$ and $\big( (X, d), x,y, \mu \big)$ with $r <d(x,y)$, the above proof still works. The first part of the proof (i.e. proving that the radii of the $\big( \overline{B_r}(X_n) \big)_{n \geq 0}$ are bounded) is not necessary anymore. We may apply the second part of the proof directly to $X$ and $X_n$ instead of the compact subsets $X'$ and $X'_n$. Note that if $r>d(x,y)$, then $\overline{B_r}(X)=X$ and $\overline{B_r}(X_n)=X_n$ for $n$ large enough, so the conclusion is immediate.
\end{proof}
\bibliographystyle{abbrv}
\bibliography{bibli}

\end{document}